\documentclass{amsart}
\usepackage{etex}
\usepackage[all, cmtip]{xy}
\usepackage{tikz,amssymb,amsmath,amscd,xy,graphicx,textcomp,mathtools}
\usepackage{epsf} 
\usepackage{titlepic}
\usepackage{hyperref}
\usepackage{tabulary}
\usepackage{booktabs}

\hypersetup{
  colorlinks   = true, 
  urlcolor     = blue, 
  linkcolor    = blue, 
  citecolor   = blue 
}

\newtheorem{theorem}{Theorem}[section]
\newtheorem{lemma}[theorem]{Lemma}

\newtheorem{definition}[theorem]{Definition}

\newtheorem{remark}[theorem]{\it Remark}
\newtheorem{example}[theorem]{Example}
\newtheorem{proposition}[theorem]{Proposition}

\xyoption{arrow}

\xyoption{matrix}

\def\C{\mathbb{C}}
\def\R{\mathbb{R}}
\def\Z{\mathbb{Z}}

\def\del{\partial}

\def\tree{\mathcal{T}}

\DeclareMathOperator{\Pic}{Pic} 
\DeclareMathOperator{\Hom}{\mathrm{Hom}}

\begin{document}

\title[Principal bundles and the $K$-Pieri rule]{Cox rings of moduli of quasi parabolic principal bundles and the $K$-Pieri rule}
\author{ Christopher Manon}

%\title{}
%\author{}

\begin{abstract}
We study a toric degeneration of the Cox ring of the moduli of quasi principal $\mathrm{SL}_m(\C)$ bundles on a marked projective line in the case where the parabolic data is chosen in the stabilizer of the highest weight vector in $\C^m$ or its dual representation $\bigwedge^{m-1}(\C^m)$.  The result of this degeneration is an affine semigroup algebra which is naturally related to the combinatorics of the $K$-Pieri rule from Kac-Moody representation theory. We find that this algebra is normal and Gorenstein, with a quadratic square-free Gr\"obner basis.  This implies that the Cox ring is Gorenstein and Koszul for generic choices of markings, and generalizes results of Castravet, Tevelev and Sturmfels, Xu. Along the way we describe a relationship between the Cox ring and a classical invariant ring studied by Weyl. 
\end{abstract}

\maketitle

\tableofcontents

\noindent
Keywords: conformal blocks, Pieri rule, toric degeneration, invariant theory

\section{Introduction}

The first fundamental theorem of invariant theory describes the algebra $R(a, b)$ of $\mathrm{SL}_m(\C)$-invariant polynomial functions on the affine space $(\bigwedge^{m-1}(\C^m))^a \times (\C^m)^b$.  The proof by Weyl \cite{Weyl} gives an elegant, simple presentation of $R(a, b)$ by quadratic relations  (see Theorem \ref{ffti}).  This sets the pattern for many of the results of classical invariant theory, and as a far-reaching application it provides the Pl\"ucker equations, which cut out one of the ubiquitous moduli spaces of algebraic geometry, the Grassmannian variety $Gr_m(\C^n).$  The main result of this paper draw on a close relationship (Theorem \ref{algebracorrelation}) between $R(a, b)$ and the total coordinate ring (a.k.a the Cox ring) $V_{\mathbb{P}^1, \vec{p}}(a, b)$ of another moduli problem, a stack of quasi-parabolic principal bundles on an $a+ b=n$-marked projective line. We show that $V_{\mathbb{P}^1, \vec{p}}(a, b)$ can be presented by quadratic relations when $\vec{p}$ is a generic arrangement of points, and moreover we find that it is Koszul and Gorenstein.   The crux of our argument is combinatorial, and emerges from so-called Pieri rules (see \cite[page 80]{FH}  ) from the representation theory of $\mathrm{SL}_m(\C)$ and the theory of conformal blocks.

Let $C$ be a smooth, projective complex curve curve and let $\vec{p} = \{p_1, \ldots, p_n\} \subset C$ be a choice of distinct marked points. We fix a Borel subgroup $B \subset \mathrm{SL}_m(\C)$, and choose a parabolic subgroup $B \subset P_i$ for each marked point $p_i$.   A quasi-parabolic principal bundle $(E, \vec{\rho})$ of type $\vec{P} = \{P_1, \ldots, P_n\}$ is an $\mathrm{SL}_m(\C)$ principal bundle $E$ on $C$ and a choice $\rho_i \in E\times_{\mathrm{SL}_m(\C)} \mathrm{SL}_m(\C)/P_i$ from the fiber of the associated $\mathrm{SL}_m(\C)/P_i$ bundle over the point $p_i.$   The moduli stack $\mathcal{M}_{C, \vec{p}}(\vec{P})$ of type $\vec{P}$ quasi-parabolic principal bundles can be thought of as a generalization of the Jacobian of $C$ which also incorporates aspects of the representation theory of $\mathrm{SL}_m(\C)$.  For example, Beauville, Laszlo and Sorger \cite{BL}, \cite{LS}, \cite{BLS}, Kumar, Narasimhan, and Ramanathan \cite{KNR}, and Pauly \cite{P} compute the Picard group of $\mathcal{M}_{C, \vec{p}}(\vec{P})$ in terms of the character lattices $\mathcal{X}_i$ of the parabolic groups $P_i$ as follows:

\begin{equation}\label{piceq}
\Pic(\mathcal{M}_{C, \vec{p}}(\vec{P})) = \mathcal{X}_1 \times \cdots \times \mathcal{X}_n \times \Z.\\
\end{equation}

The total coordinate ring $V_{C, \vec{p}}(\vec{P})$ of $\mathcal{M}_{C, \vec{p}}(\vec{P})$ is the  $\Pic(\mathcal{M}_{C, \vec{p}}(\vec{P}))$-graded sum of all the global section spaces $H^0(\mathcal{M}_{C, \vec{p}}(\vec{P}), \mathcal{L}(\vec{\lambda}, K))$, $\vec{\lambda} = \{\lambda_1, \ldots, \lambda_n\}$, $\lambda_i \in \mathcal{X}_i$, $K \in \Z,$ with product given by global section multiplication. A celebrated result of several authors (see \cite{KNR}, \cite{BLS}, \cite{P}, and \cite{Fal}) identifies each global section space $H^0(\mathcal{M}_{C, \vec{p}}(\vec{P}), \mathcal{L}(\vec{\lambda}, K))$ with a space of so-called conformal blocks $V_{C, \vec{p}}(\vec{\lambda}, K)$ from the Wess-Zumino-Novikov-Witten model of conformal field theory associated to the Lie algebra $sl_m(\C)$.    In this way, $V_{C, \vec{p}}(\vec{P})$ is interesting both as an object from mathematical physics and as a generalization of the algebra of theta functions. 

We describe the total coordinate ring in a special case $\mathcal{M}_{\mathbb{P}^1, \vec{p}}(\vec{P}, \vec{P^*}),$ where the curve $C$ is a projective line, and the parabolic structure at each marked point is given by the stabilizers $P^*, P \subset \mathrm{SL}_m(\C)$ of the highest weight vector in the representation $\C^m$ or its dual representation $\bigwedge^{m-1}(\C^m)$ respectively.   We let $a$ be the number of marked points with parabolic data $P$, $b$ be the number of marked points with parabolic data $P^*$, and we let $V_{\mathbb{P}^1, \vec{p}}(a, b)$ denote the total coordinate ring of this stack. Our main result is the following theorem.

\begin{theorem}\label{main}
For generic arrangements of points $\vec{p},$ the algebra $V_{\mathbb{P}^1, \vec{p}}(a, b)$ is a Koszul, Gorenstein algebra, and is minimally generated by the conformal blocks with $K = 1$. In particular, the ideal of relations on these conformal blocks is quadratically generated.
\end{theorem}

Recall that the Koszul property means that the field of scalars $\C$ has a minimal graded free resolution when regarded as a $V_{\mathbb{P}^1, \vec{p}}(a, b)$ algebra (see \cite[Chapter 2, Definition 1]{PP}. As a Gorenstein algebra the module $Ext_S^{N-d}(V_{\mathbb{P}^1, \vec{p}}(a, b), S)$ is isomorphic to grade-shifted copy of $V_{\mathbb{P}^1, \vec{p}}(a, b)$ (\cite[Theorem 3.3.7]{BH}), where $N$ is the sum of the dimensions of the spaces of conformal blocks with $K = 1$, $S$ is a polynomial algebra on $N$ variables which presents $V_{\mathbb{P}^1, \vec{p}}(a, b)$, and $d$ is the Krull dimension of $V_{\mathbb{P}^1, \vec{p}}(a, b)$.  This implies that the Betti numbers of $V_{\mathbb{P}^1, \vec{p}}(a, b)$ as an $S$ module satisfy a type of Poincare duality: $\beta_i^S = \beta_{N-d -i}^S.$  The degree $A$ of the grade shift  $Ext_S^{N-d}(V_{\mathbb{P}^1, \vec{p}}(a, b), S) \cong V_{\mathbb{P}^1, \vec{p}}(a, b)[A]$ is called the $A$-invariant. We prove (Proposition \ref{interlacegor}) that $A = -2m$ when $a$ and $b$ are larger than $m$.

When $m = 2,$ the parabolic subgroups $P, P^*$ are equal, this case is studied by Castravet and Tevelev in \cite{CT}, where the algebra $V_{\vec{p}}(n)$ is shown to be a Cox-Nagata ring.  This means that $V_{\vec{p}}(n)$ is the algebra of invariants by a certain non-reductive group action on a polynomial ring, and can be identified with the Cox ring of a blow-up of a projective space.  Sturmfels and Xu also study this case in \cite{StXu}, where they construct a SAGBI degeneration $V_{\tree}(n)$ of $V_{\vec{p}}(n)$ for generic $\vec{p}$.   Here $\tree$ is a piece of combinatorial data, it denotes a trivalent tree with $n$ ordered leaves.  The degenerations $V_{\tree}(n)$ constructed in \cite{StXu} are affine semigroup algebras (see Section \ref{semigroup-algebras-section}) are also studied by  Buczy\'{n}ska and Wi\'{s}niewski \cite{BW} in the context of mathematical biology.  Buczy\'{n}ska and Wi\'{s}niewski prove results for $V_{\tree}(n)$ which imply Theorem \ref{main} in the $\mathrm{SL}_2(\C)$ case, by way of the degeneration constructed in \cite{StXu}.

\subsection{Methods and outline of the paper}

We approach $V_{\mathbb{P}^1, \vec{p}}(a, b)$ using a similar recipe of degeneration developed for the general case in \cite{M4}.  For a general
selection of parabolic subgroups $\vec{P}$, \cite[Theorem 1.1]{M4} shows that $V_{\mathbb{P}^1, \vec{p}}(\vec{P})$ can be flatly degenerated to 
any member of a class of algebras $V_{\tree}(\vec{P})$, where once again $\tree$ is a trivalent tree with $n$ ordered leaves (see Section \ref{abdegen}).   The algebras $V_{\tree}(\vec{P})$ are not affine semigroup algebras in general, however they are in a certain sense ``closer'' to affine semigroup algebras than $V_{\mathbb{P}^1, \vec{p}}(\vec{P}).$  Up to automorphism, there is only one $3$-marked projective line $(\mathbb{P}^1, 0, 1, \infty)$, so we may  speak of the unique total coordinate ring $V_{0, 3}(P_1, P_2, P_3)$ in this case.  The algebra $V_{\tree}(\vec{P})$ is the algebra of invariants in a tensor product of these $3$-marked algebras with respect to an algebraic torus, where the individual algebras $V_{0, 3}(P_1, P_2, P_3)$ involved and the torus action depend on $\tree$ and the choice $\vec{P}$.   Because of this format, $V_{\tree}(\vec{P})$ is an affine semigroup algebra if each $V_{0, 3}(P_1, P_2, P_3)$ is an affine semigroup algebra.  To prove Theorem \ref{main} we choose a special tree $\tree_0$ (see Figure \ref{cat}) to guarantee this is the case for each of the four algebras $V_{0, 3}(P, P, B),$ $V_{0, 3}(B, P, B)$, $V_{0, 3}(B, P^*, P^*)$ and $V_{0, 3}(B, P^*, B)$ which appear in the construction of the associated algebra $V_{\tree_0}(a, b).$   From now on we refer to these four algebras as the $K$-Pieri algebras. 

The $K$-Pieri rule (Theorem \ref{kpieri}) enters the picture in our analysis of $V_{0, 3}(P, P, B),$ $V_{0, 3}(B, P, B)$, $V_{0, 3}(B, P^*, P^*)$ and $V_{0, 3}(B, P^*, B)$.   The $K$-Pieri rule implies that the $\Pic$-graded components $V_{0, 3}(\lambda, \mu, \eta, K)$ of each of these algebras are multiplicity-free, and gives necessary and sufficient conditions, in the form of explicit inequalities on the $\mathrm{SL}_m(\C)$-weights $\lambda, \mu, \eta$ and the level $K,$ for a component to be non-zero.  We use these inequalities to show that the $K$-Pieri algebras are polynomial rings (Proposition \ref{4algebrasquantum} ).    In Section \ref{abdegen}, Theorem \ref{kpresent} we use a toric fiber product argument (see Section \ref{semigroup-algebras-section}) to show that $V_{\tree_0}(a, b)$ is presented by an ideal with a quadratic, square-free Gr\"obner basis, this in turn implies that this algebra is Koszul (see e.g. \cite[Theorem 3.1]{PP}).   The algebra $V_{\mathbb{P}^1, \vec{p}}(a, b)$ is therefore Koszul by general properties of flat degenerations (see e.g. \cite[Theorem 1.11]{M6}).

 The techniques we use on $V_{\mathbb{P}^1, \vec{p}}(a, b)$ can also be applied to Weyl's algebra of invariants $R(a, b)$.    The algebra $R(a, b)$ has a flat degeneration to an algebra $R_{\tree_0}(a, b)$ (Proposition \ref{semiiso}), which is built from four algebras $R(P, P, B),$ $R(B, P, B)$, $R(B, P^*, B)$, and $R(B, P^*, P^*)$.  Each of these ``classical Pieri algebras'' is closely related to its counterpart $K$-Pieri algebra (Theorem \ref{algebracorrelation}).  Using the Pieri rule from the representation theory of $\mathrm{SL}_m(\C)$ (Proposition \ref{p2}), we prove that each of these ``classical Pieri algebras'' is a polynomial ring in Proposition \ref{classicalaffinesemigroup}, this is the crucial ingredient in the proof of Theorem \ref{kpresent}.  

In Section \ref{interlacing} we view the algebras $R_{\tree_0}(a, b)$ and $V_{\tree_0}(a, b)$ from two more perspectives.  We show that the affine semigroups underlying both of these algebras are isomorphic to semigroups of interlacing patterns, also known as Gel'fand-Tsetlin patterns.  This alternative point of view allows us to prove that these algebras (and therefore their deformations $R(a, b)$ and $V_{\mathbb{P}^1, \vec{p}}(a, b)$) are Gorenstein, completing the proof of Theorem \ref{main}.  We also give explicit presentations of $R_{\tree_0}(a, b)$ and $V_{\tree_0}(a, b)$ by generators and relations.  Finally, we note that much of what we say in this paper goes through without modification for the moduli stack $\mathcal{M}_{\mathbb{P}^1, \vec{p}}(B, \vec{P}, \vec{P}^*, B)$, see Remark \ref{2gen}.

\subsection{Acknowledgements}

We thank Avinash Dalal, Jennifer Morse, and Kaie Kubjas for enlightening conversations in the course of this project, and we thank the reviewers for helpful suggestions.

\section{Affine semigroup algebras}\label{semigroup-algebras-section}

A flat degeneration to an affine semigroup algebra is a powerful investigative tool
because many algebraic properties of affine semigroup algebras can be proved combinatorially.  For a semigroup $P$ we let $\C[P]$ be the semigroup algebra with complex coefficients, and for a collection of elements $x_1, \ldots, x_k \in P$, we let $x_1 + \cdots + x_k$ denote the sum in $P$ and $[x_1]\cdots[x_k] = [x_1 + \cdots + x_k]$ denote the associated monomial in $\C[P].$    Affine semigroups have a fiber product operation called the toric fiber product (see \cite{Sullivant} and \cite{M6}).  The word ``toric'' is used in this context because the fiber product operation on two semigroups $P, Q$ corresponds to taking a certain subalgebra of torus invariants in $\C[P\times Q] = \C[P]\otimes \C[Q].$  We review the stability of presentation data under the toric fiber product operation, and we recall combinatorial conditions which imply that an affine semigroup algebra is Gorenstein.

\begin{example}\label{exampleinterlace}
An important example of a rational polyhedral cone $\Delta_n$ with a freely generated affine semigroup $\Lambda_n$ will appear in the following section, namely the set of weakly decreasing $n$-tuples of non-negative real numbers: 

\begin{equation}
\Delta_n = \{ (\lambda_1, \ldots, \lambda_n) \ | \ \lambda_i \geq \lambda_{i+1} \geq 0,~ 1 \leq i \leq n-1\}.\\
\end{equation}

\noindent
The affine semigroup $\Lambda_n = \Delta_n \cap \Z^n$ is the set of all partitions of length at most $n$.  For any $\lambda = (\lambda_1, \ldots, \lambda_n) \in \Lambda_n$ there is a unique decomposition:

\begin{equation}
\lambda = (\lambda_1 - \lambda_2)\mathbf{\omega}_1+ \cdots (\lambda_i - \lambda_{i+1})\mathbf{\omega}_i + \cdots  + \lambda_n\mathbf{\omega}_n,\\
\end{equation}

\noindent
where $\mathbf{\omega}_i \in \Lambda_n$ is the element given by $i$ $1$'s followed by $n-i$ $0$'s.  The elements $\mathbf{\omega}_i$ are known as the fundamental weights of $\Delta_n$ when it is viewed as a Weyl chamber of $\mathrm{SL}_{n+1}(\C)$.  The affine semigroup algebra $\C[\Lambda_n]$ is a polynomial ring on $n$ variables.
\end{example}

\subsection{Fiber products}

A convex rational polyhedral cone $\mathcal{P}$ is a convex subset of $\R^n$ defined as the intersection of a finite number of of half spaces which contain the origin and have a rational normal vector with respect to the standard basis $e_1, \ldots e_n \in \R^n$.  For our purposes a map $\pi: \mathcal{P} \to \mathcal{D}$ of rational polyhedral cones with $\mathcal{P} \subset \R^n$ and $\mathcal{D} \subset \R^k$ is induced from a linear map $\pi: \R^n \to \R^k$ whose associated matrix is rational.

\begin{definition}\label{fiberproduct}
For $\mathcal{P}_1 \subset \R^n$, $\mathcal{P}_2 \subset \R^m$, $\mathcal{D} \subset \R^k$ rational polyhedral cones, and maps $\pi_i: \mathcal{P}_i \to \mathcal{D}$, the toric fiber product is the following set:

\begin{equation}
\mathcal{P}_1 \times_{\mathcal{D}} \mathcal{P}_2 = \{(x, y) \ | \ \pi_1(x) = \pi_2(y)\} \subset \mathcal{P}_1 \times \mathcal{P}_2 \subset \R^n \times \R^m.\\
\end{equation}

\end{definition}
 
\noindent
It is straightforward to check that the toric fiber product $\mathcal{P}_1 \times_{\mathcal{D}} \mathcal{P}_2$ is also a polyhedral cone in $\R^n \times \R^m.$ 

A normal affine semigroup is of the form $\mathcal{P} \cap \Z^n \subset \R^n$ for $\mathcal{P}$ a polyhedral cone, all of the semigroups we encounter in this paper of this type.  If we further assume that the rational maps $\pi_i$ in Definition \ref{fiberproduct} are integral, then it makes sense to write $\pi_i: P_i \to D = \mathcal{D} \cap \Z^k$, and we can consider the fiber product semigroup:

\begin{equation}
P_1 \times_D P_2 = \{ (x, y) \in P_1 \times P_2 \ | \ \pi_1(x) = \pi_2(y) \in D\}.\\
\end{equation}

If the generators of the affine semigroups behave well with respect
to the maps $\pi_1, \pi_2,$ then the affine semigroup of the associated toric fiber
product is also well-behaved. The next proposition is \cite[Proposotion 3.1]{M6}.

\begin{proposition}\label{semigen}
Let $\mathcal{P}_i, \mathcal{D}$ be as in Definition \ref{fiberproduct} with $\pi_i: \mathcal{P}_i \to \mathcal{D}$ integral maps.  Assume $D = \mathcal{D}\cap \Z^k$ is freely generated by the subset $T \subset D$, and let $S_i \subset P_i$ be a generating set. Furthermore, suppose that $\pi_i(S_i) \subset T \cup \{0\}$, then the fiber product set $S_1 \cup \{0\} \times_{T\cup \{0\}} S_2 \cup \{0\}$ generates the fiber product semigroup $P_1 \times_D P_2.$
\end{proposition}

\begin{proof}
For any semigroup element $(x, y) \subset P_1 \times_D P_2$, we can factor $x = s_1 + \cdots + s_a$, $s_i \in S_1$, and $y = t_1 + \cdots + t_b$, $t_j \in S_2$ by assumption, this gives the following expression in elements of $D$: 

\begin{equation}
\pi_1(x) = \pi_2(y) = \pi_1(s_1) + \cdots + \pi_1(s_a) = \pi_2(t_1)+ \cdots + \pi_2(t_b).\\
\end{equation}

Assume without loss of generality that $a \leq b$. The affine semigroup $D$ is freely generated by $T$, it follows that each element $\pi_1(s_i) \in T$ also appears as some $\pi_2(t_i) \in T$ (we may reorder the $t_j$ so that this indexing works out), and that the remaining $\pi_2(t_{\ell})$ are all equal to $0$.  This implies that the following pairs are all in the fiber product $S_1 \times_T S_2$:

\begin{equation}
(s_1, t_1), \ldots, (s_a, t_a), (0, t_{a+1}), \ldots, (0, t_b),\\
\end{equation}

\noindent
and furthermore that we have a factorization of $(x, y)$:

\begin{equation}
(s_1, t_1) + \cdots + (s_a, t_a) + (0, t_{a+1}) + \cdots + (0, t_b) = \Big(\sum s_i, \sum t_j\Big) = (x, y).\\
\end{equation}

\end{proof}

The group $\Z^n$ can be identified with the characters $\chi: (\C^*)^n \to \C^*$ of an $n$-torus, and likewise so can the elements of an affine semigroup $P \subset \Z^n$.  In this way, the algebra $\C[P_1 \times P_2]$ naturally comes with an algebraic action of $(\C^*)^{n+m}$. The linear maps $\pi_i: P_i \to D$ define dual maps $\pi_1^*: (\C^*)^k \to (\C^*)^n, \pi_2^*: (\C^*)^k \to (\C^*)^m$, using these we define  a subtorus $T \subset (\C^*)^{n+m}$ to be the image of $(\pi_1^*, [\pi_2^*]^{-1}).$ The invariant algebra $\C[P_1 \times P_2]^T$ is the subalgebra of $\C[P_1 \times P_2]$ with basis given by those points $(x, y)$ with $\pi_1(x) - \pi_2(y) = 0$, namely $\C[P_1 \times_D P_2].$

In what follows, let $S_X$ be the polynomial ring on variables associated to the members of a set $X.$ The sets $S_1 \subset P_1, S_2 \subset P_2$ and $T \subset D$ from Proposition \ref{semigen} define presentations of the corresponding semigroup algebras by polynomial rings: $S_{S_1} \to \C[P_1], S_{S_2} \to \C[P_2]$, and $S_T \to \C[D]$, in particular $\C[D]$ is isomorphic to $S_T$ by the assumption that it is freely generated.  We let $I_{P_1} \subset S_{S_1}, I_{P_2} \subset S_{S_2}$ be the binomial ideals which vanish on the presentation map. Proposition \ref{semigen} implies that under the stated assumptions, these presentations can be used to construct a presentation of $\C[P_1\times_D P_2]$ by the polynomial algebra $S_{S_1\times_T S_2}$.  Next we review results which control the behavior of the corresponding binomial ideal $I_{P_1\times_D P_2}$ of this presentation.

In the proof of Proposition \ref{semigen} several of the elements $\{s_1, \ldots, s_i\}$ may have the same image under $\pi_1$, leading to distinct assignments $s_k \to t_k$ which produce different factorizations of the element $(x, y),$ these factorizations satisfy relations constructed as follows.  For $a, b \in P_1,$ $c, d \in P_2$ with $\pi_1(a) = \pi_2(c) = \pi_1(b) = \pi_2(d)$ an equation $(a, c)+ (b, d) = (b, c) + (a, d)$ holds in $P_1 \times_D P_2,$ we call these swap relations.

For what follows we refer the reader to the book \cite{St} for the basics of Gr\"obner bases.  If $P_1, D, P_2$ are graded semigroups with $S_1, T, S_2$ sets of degree $1$ elements,  then Proposition \ref{semigen} implies that $P_1\times_D P_2$ is also generated in degree $1$. This sort of stability of generating degrees under toric fiber product can also be extended to the defining ideal $I_{P_1\times_D P_2}$ and to Gr\"obner bases of these ideals.

\begin{proposition}\label{kos}
Let $P =\Lambda_{k_1}\times_{\Lambda_{m_1}} \cdots \times_{\Lambda_{m_N}}\Lambda_{k_{N+1}}$ be a fiber product of free graded affine semigroups  generated in degree $1.$  We assume the maps defining this fiber product send the generating set of each factor to a generator or the identity in the appropriate base.   Then the affine semigroup algebra $\C[P]$ is generated in degree $1$, and the swap relations define a quadratic, square-free Gr\"obner basis on the presenting ideal $I_P$. 
\end{proposition}

\begin{proof}
The generation portion of this statement follows by induction from Proposition \ref{semigen}. To address relations, we proceed by induction and adapt the arguments used in \cite[Section 2]{Sullivant} and \cite[Section 4]{M6}.   Let $S_1 \subset \Lambda_{m_1}, S_2 \subset \Lambda_{m_2}$, and $T \subset \Lambda_{k}$ be generating sets so that the set of pairs $(s, t) \in S_1\cup \{0\} \times S_2\cup\{0\}$ where $\pi_1(s) = \pi_2(t)$ generates $\C[\Lambda_{m_1}\times_{\Lambda_{k}} \Lambda_{m_2}]$. Let $Z_1 \subset S_1$ and $Z_2 \subset S_2$ be the sets of generators sent to $0$ by the maps $\pi_1, \pi_2$, then $\C[\Lambda_{m_1}\times_{\Lambda_{k}} \Lambda_{m_2}]$ is polynomial ring in variables $Z_1 \cup Z_2$ over $\C[\Lambda_{m_1'}\times_{\Lambda_{k}} \Lambda_{m_2'}]$, where $m_i = |S_i| - |Z_i|$.  In this way we may reduce to the case where the maps $\pi_i$ only send elements of $S_i$ to $T$.  In this case the Proposition follows from \cite[Proposition 10]{Sullivant} and \cite[Proposition 4.1]{M6}, in particular Sullivant's relations in the statement of Proposition 10 are the swap relations.  

Now we suppose the proposition has been proved for $t \leq N$.  We write $P = \Big(\Lambda_{k_1}\times_{\Lambda_{m_1}} \cdots \times_{\Lambda_{m_{N-1}}}\Lambda_{k_{N}}\Big) \times_{\Lambda_{m_N}} \Lambda_{k_{N+1}} = P'\times_{\Lambda_{m_N}} \Lambda_{k_{N+1}}$. By assumption there is a term order $<'$ on the monomials in the generators of $\C[P']$ which make the swap relations a quadratic, square-free Gr\"obner basis for $I_{P'}$.  We concoct a term order $<$ on the generators of $\C[P]$ by ordering first with $<'$ and then breaking ties with a chosen ordering on the generators of $\Lambda_{k_{N+1}}$.  This is extended lexicographically to the monomials in these generators.  Recall that a monomial $[s_1, v_1] \cdots [s_k,v_k]$ is said to be standard with respect to $<$ if for all monomials $[t_1, u_1]\cdots[t_k,u_k]$ with  $[s_1, v_1]\cdots[s_k,v_k] - [t_1, u_1]\cdots[t_k,u_k] \in I_P$ we have  $[s_1, v_1]\cdots[s_k,v_k] < [t_1, u_1]\cdots[t_k,u_k]$.  To prove the proposition it suffices to show that if $[s_1, v_1]\cdots[s_k,v_k]$  is not standard, it can be changed to a lower monomial with the same image in $\C[P]$ by performing a swap relation.   

Suppose that $[s_1, v_1]\cdots[s_k,v_k]$ (written from greatest generator to least) is not standard, so that there is some monomial $ [t_1, u_1]\cdots[t_k,u_k]$ with  $[s_1, v_1]\cdots[s_k,v_k] > [t_1, u_1]\cdots[t_k,u_k]$. Suppose that the underlying $P'$ monomials are not equal, then by definition $[s_1]\cdots[s_k] >' [t_1]\cdots[t_k]$.  By induction, there is a swap relation $[s_i][s_j] - [s_{i'}][s_{j'}] \in I_{P`}$ with $[s_i][s_j] > [s_{i'}][s_{j'}]$.  Furthermore, as $\Lambda_{m_n}$ is freely generated, the sets $\{\pi_1(s_i), \pi_1(s_j)\}$ and $\{\pi_1(s_{i'}), \pi_1(s_{j'})\}$ are the same. It follows that we can replace $s_i, s_j$ with $s_{i'}, s_{j'}$ by one of the following swap relations in the expression for  $[s_1, v_1]\cdots[s_k,v_k]$ and get a lower monomial: $[s_i, v_i][s_j, v_j] - [s_{i'}, v_i][s_{j`}, v_j]$,  $[s_i, v_i][s_j, v_j] - [s_{i'}, v_j][s_{j`}, v_i]$ .   Either of these is a swap relation in $I_P$, so we can therefore assume that $s_i = t_i$ for all $i$.  Let $[s_i, v_i] > [t_i, u_i]$ be the first place these monomials differ, then we must have $v_i > u_i$ in the ordering on the generators of $\Lambda_{k_{N+1}}$, and $\pi_2(u_i) = \pi_2(v_i)$ since $s_i = t_i$.  As $\Lambda_{k_{N+1}}$ is freely generated, we must have $[s_j, u_i]$ in the monomial for some $j > i$ and $[s_i] >'[s_j]$.  It follows that we may perform the swap $[s_i, v_i][s_j, u_i] - [s_i, u_i][s_j, v_i]$ to lower the monomial.

\end{proof}

\subsection{The Gorenstein property}

For a polyhedral cone $\mathcal{P}$ we let $int(\mathcal{P})$ denote the set of (relative) interior points, namely those points which do not belong to a proper facet of $\mathcal{P}.$  The set $int(\mathcal{P}) \cap P = int(P)$ spans a proper ideal of $\C[P]$ which can be used to characterize the Gorenstein property. The following proposition is \cite[Corollary 6.3.8]{BH}.

\begin{proposition}\label{gor}
Let $P$ be a normal affine semigroup.  The algebra $\C[P]$ is Gorenstein if and only if 
$int(P) = w + P$ for some $w \in int(P).$  In the presence of a grading, the $A$-invariant of $\C[P]$ is $-deg(w).$
\end{proposition}

\begin{example}
The affine semigroup algebra $\C[\Lambda_n]$ from Example \ref{exampleinterlace} is a polynomial ring, and is therefore Gorenstein. The interior points $int(\Lambda_n)$ are generated by $w_n = (n, n-1, \ldots, 1)$. 
\end{example}

\section{Representation theory and the Pieri algebras}\label{rep-theory-section}

  The classification of irreducible representations and the rules governing behavior of tensor products of these representations, including the Pieri rule, are written in the language of convex polyhedra and affine semigroups. After reviewing these topics we give a proof of Proposition \ref{4algebrasclassical}, which says that each of the four Pieri algebras $R(P, P, B)$, $R(B, P, B)$, $R(B, P^*, B)$ and $R(B, P^*, P^*)$ is a polynomial ring.  We refer the reader to the books of Fulton and Harris \cite{FH} and Humphreys \cite{Humphreys} for more background on representation theory. 

\subsection{Irreducible representations of $\mathrm{GL}_m(\C)$ and $\mathrm{SL}_m(\C)$}

We begin by recalling some of the consequential subgroups for the representation theory of $G = \mathrm{SL}_m(\C)$ or $\mathrm{GL}_m(\C)$.  We let $T_G$ and $B_G$ denote the maximal diagonal torus and the Borel subgroup of upper triangular matrices, respectively. The parabolic subgroups containing $B_G$ are denoted by $P \subset G$.  The unipotent radical of $B_G$ is denoted by $U_G \subset G$, this is the subgroup of upper triangular matrices with $1$'s along the diagonal. When it is clear from context we drop the $G$ subscript and write $T, B, U$ for these subgroups. 

 Irreducible representations of $G$ are indexed by special characters of the maximal torus $\lambda: T_G \to \C^*$ called dominant weights.   The group of characters $\mathcal{X}(T_G)$ is isomorphic to $\Z^m$ when $G = \mathrm{GL}_m(\C)$ and $\Z^{m-1}$ when $G = \mathrm{SL}_m(\C)$.  Dominant weights are precisely the integral points inside of a convex polyhedral cone $\Delta_G \subset \mathcal{X}(T_G)\otimes \R$ called the Weyl chamber.  We use the standard convention that a point  $ \lambda \in \R^m$  is in $\Delta_{\mathrm{GL}_m(\C)}$ if and only if its entries are weakly decreasing:

\begin{equation}
\lambda_1 \geq \lambda_2 \geq \cdots \geq \lambda_m.\\
\end{equation}

\noindent
The affine semigroup of dominant weights is denoted $\Lambda_{\mathrm{GL}_m(\C)}$. In analogy to Example \ref{exampleinterlace}, any $\lambda \in \Lambda_{\mathrm{GL}_m(\C)}$ can be written uniquely as a sum $\ell_1 \mathbf{\omega}_1 + \cdots + \ell_m \mathbf{\omega}_m$ with $\ell_i$ non-negative for $1  \leq i \leq m-1$ and $\ell_m \in \mathbb{Z}$ .  The irreducible representation corresponding to a dominant weight $\lambda \in \Lambda_{\mathrm{GL}_m(\C)}$ is denoted $V(\lambda),$ in particular $V(\mathbf{\omega}_i)$ is the exterior power $\bigwedge^i(\C^m)$ with its standard action by $\mathrm{GL}_m(\C).$ 

Much of the representation theory of $\mathrm{SL}_m(\C)$ can be deduced from that of $\mathrm{GL}_m(\C)$ by the inclusion $\mathrm{SL}_m(\C) \subset \mathrm{GL}_m(\C).$ Notice that $V(\omega_m)$ is the determinant representation of $\mathrm{GL}_m(\C)$ and therefore the trivial representation when restricted to $\mathrm{SL}_m(\C).$  More generally, two irreducible representations $V(\lambda)$ and $V(\eta)$, $\lambda, \eta \in \Delta_{\mathrm{GL}_m(\C)}$ restrict to the same irreducible representation on $\mathrm{SL}_m(\C)$ if and only if $\lambda = \eta + N\mathbf{\omega}_m$ for some $N \in \Z$ (see \cite[15.5]{FH}).   All irreducible representations of $\mathrm{SL}_m(\C)$ can be constructed this way, so it is natural to use the conventional choice of Weyl chamber $\Delta_{\mathrm{SL}_m(\C)} \subset \R^{m-1} \cong \R^m/\R(1, \ldots, 1)$ consisting of $m-1$ tuples of weakly descreasing non-negative real numbers:

\begin{equation}
\lambda_1 \geq \cdots \geq \lambda_{m-1} \geq 0.\\
\end{equation}

\noindent
The weights $\Lambda_{\mathrm{SL}_m(\C)}$ are then the $m-1$ tuples of weakly decreasing non-negative integers, and the fundamental weights of $\Lambda_{\mathrm{SL}_m(\C)}$ are $\mathbf{\omega}_1, \ldots, \mathbf{\omega}_{m-1}$.   The Weyl chamber $\Delta_{\mathrm{SL}_m(\C)}$ is a simplicial cone, with extremal rays generated by the $\mathbf{\omega}_i$. Each face $F \subset \Delta_{\mathrm{SL}_m(\C)}$ of this cone corresponds to a distinguished parabolic subgroup $B \subset P_F \subset \mathrm{SL}_m(\C)$, where $F$ consitutes those dominant weights whose associated characters of $T_{\mathrm{SL}_m(\C)}$ extend to $P_F.$  In particular, we let $P^*$ be the parabolic corresponding to the ray through $\mathbf{\omega}_{m-1}$ and $P$ be the parabolic corresponding to the ray through $\mathbf{\omega}_1.$

For any irreducible representation $V(\lambda)$, the induced representation on the dual vector space $V(\lambda)^* \cong V(\lambda^*)$ is also irreducible, and corresponds to a dominant weight $\lambda^*$.  The duality map on weights $d: \Lambda \to \Lambda$, $d(\lambda) = \lambda^*$ is known to be induced by a linear map $d: \Delta \to \Delta$ on the ambient Weyl chamber. Duality acts on $\Delta_{\mathrm{SL}_m(\C)}$ by switching $\mathbf{\omega}_i$ with $\mathbf{\omega}_{m-i}$.

\subsection{Tensor products and the classical Pieri rule}\label{pr}

A tensor product of any pair of irreducible representations of a reductive group $G$ has a unique decomposition into irreducible representations with multiplicity:

\begin{equation}\label{tdecomp}
V(\lambda) \otimes V(\eta) = \bigoplus_{\mu \in \Lambda_{G}} \Hom_{G}(V(\mu), V(\lambda) \otimes V(\eta))\otimes V(\mu).\\
\end{equation}

\noindent
Here $\Hom_{G}(V(\mu), V(\lambda) \otimes V(\eta))$ is the space of interwiners $\rho: V(\mu) \to V(\lambda) \otimes V(\eta).$   In the $\mathrm{GL}_m(\C)$ case, the Pieri rule (\cite[Exercise 6.12]{FH}) gives a recipe for the decomposition in Equation \ref{tdecomp} when one of the factors is a multiple of the first fundamental weight $\mathbf{\omega}_1$.   Recall that two weights $\lambda$ and $\eta$ are said to interlace, written $\eta \prec \lambda$, if $\lambda_i - \eta_i \geq 0$ and $\\eta_i - \lambda_{i+1} \geq 0$.  

\begin{proposition}
For $\eta, \lambda \in \Lambda_{\mathrm{GL}_m(\C)}$, the representation $V(\lambda)$ appears
in the decomposition of $V(\eta) \otimes V(r\mathbf{\omega}_1)$ if and only if $\eta \prec \lambda$ and $\sum \lambda_j - \sum \eta_i = r.$
If this is the case then $\Hom_{\mathrm{GL}_m(\C)}(V(\lambda), V(\eta) \otimes V(r\mathbf{\omega}_1)) \cong \C.$
\end{proposition}

We need to implement the Pieri rule on representations of $\mathrm{SL}_m(\C)$, this is handled by the following Proposition. 

\begin{proposition}\label{p2}
For $\eta, \lambda \in \Lambda_{\mathrm{SL}_m(\C)}$, the representation $V(\lambda)$ appears in the decomposition of $V(\eta) \otimes V(r\mathbf{\omega}_1)$ if and only if there is a dominant $\mathrm{GL}_m(\C)$ weight $\bar{\lambda} \in \Lambda_{\mathrm{GL}_m(\C)}$ such that the following hold:

\begin{enumerate}
\item $\lambda = \bar{\lambda} - \bar{\lambda}_m\mathbf{\omega}_m$,\\
\item $\sum \bar{\lambda}_j - \sum \eta_i = r$,\\
\item $\eta \prec \bar{\lambda} $.\\
\end{enumerate}

\noindent
In the case that $(1) - (3)$ are satisfied, $\Hom_{\mathrm{SL}_m(\C)}(V(\lambda), V(\eta) \otimes V(r\mathbf{\omega}_1)) \cong \C.$
\end{proposition}

\begin{proof}
We can regard $\eta$ and $\lambda$ as members of $\Lambda_{\mathrm{GL}_m(\C)}$ and form the decomposition $V(\eta) \otimes V(r\mathbf{\omega}_1) = \oplus V(\bar{\lambda})$ using the Pieri rule.  Each $\bar{\lambda}$ appearing in this decomposition must satisfy $(2)$ and $(3)$, $\lambda = \bar{\lambda} - \bar{\lambda}_m\mathbf{\omega}_m$ is a dominant weight of $\mathrm{SL}_m(\C)$,  and every representation in the decomposition of $V(\eta)\otimes V(r\mathbf{\omega}_1)$ as a representation of $\mathrm{SL}_m(\C)$ must appear this way.  It remains only to check that the decomposition is multiplicity-free, so we must show that $\bar{\lambda}$ is uniquely determined by $\lambda, \eta$  and $r.$  This follows from the observation that $r = \sum \bar{\lambda}_i - \sum \eta_j = \sum (\lambda_i + \bar{\lambda}_m) - \sum \eta_j$, therefore $m\bar{\lambda}_m = r + \sum \eta_j  - \sum \lambda_i.$
\end{proof}

For general reductive $G$, a space of intertwiners is naturally isomorphic to the space of invariant vectors in a triple tensor product of irreducible representations:

\begin{equation}
\Hom_{G}(V(\mu), V(\lambda) \otimes V(\eta)) \cong [V(\mu^*)\otimes V(\lambda) \otimes V(\eta)]^{G}.\\
\end{equation}

\noindent
Applying this fact, we find that the invariant space $[V(\lambda^*) \otimes V(r\mathbf{\omega}_1) \otimes V(\eta)]^{\mathrm{SL}_m(\C)}$ is isomorphic to $\C$ or $0$, and the former occurs precisely when $\lambda, \eta,$ and $r$ satisfy the hypotheses of Proposition \ref{p2}. The tensor product operation is commutative, it follows that if $(\lambda', \mu', \eta') = \sigma(\lambda, \mu, \eta)$ for any permutation $\sigma \in \mathcal{S}^3$, the corresponding spaces of invariants are naturally isomorphic.  Furthermore, the spaces $[V(\lambda)\otimes V(\mu) \otimes V(\eta)]^{G}$ and $[V(\lambda^*)\otimes V(\mu^*) \otimes V(\eta^*)]^{G}$ are dual vector spaces.  As a consequence of these observations, we derive a dual Pieri rule for $\mathrm{SL}_m(\C).$

\begin{proposition}\label{dp}
For $\eta, \lambda \in \Lambda_{\mathrm{SL}_m(\C)}$, the space $[V(\lambda^*) \otimes V(s\mathbf{\omega}_{m-1}) \otimes V(\eta)]^{\mathrm{SL}_m(\C)}$ is non-trivial if and only if there is a dominant $\mathrm{GL}_m(\C)$ weight $\bar{\eta} \in \Lambda_{\mathrm{GL}_m(\C)}$ such that the following hold:

\begin{enumerate}
\item $\eta = \bar{\eta} - \bar{\eta}_m \mathbf{\omega}_m$,\\
\item $\sum \bar{\eta}_j - \sum \lambda_i = s$,\\
\item $\lambda \prec \bar{\eta}$.\\
\end{enumerate}

In the case that $(1)-(3)$ are satisfied, this space is isomorphic to $\C.$
\end{proposition}

Weights  satisfying the conditions of Proposition \ref{p2} can be arranged into interlacing patterns, as depicted in Figure \ref{pieri2}.  An arrow $x \to y$ between two quantities and in the pattern indicates a weak inequality $x \leq y$, compare this to condition $(3)$ in Proposition \ref{p2}.

\begin{figure}[htbp]
\begin{tikzpicture}
%The interlacing pattern
\draw [->] (0,0) -- (-.8, 1.2);
\draw [->] (-1,1.2) -- (-1.8, 0);
\draw [->] (-2, 0) -- (-2.8, 1.2);
\draw [->] (-3,1.2) -- (-3.8, 0);
\draw [->] (-4, 0) -- (-4.8, 1.2);
\draw [->] (-5,1.2) -- (-5.8, 0);
\draw [->] (-6, 0) -- (-6.8, 1.2);
\draw [->] (-7,1.2) -- (-7.8, 0);
\draw [->] (-8, 0) -- (-8.8, 1.2);
%bottom row, right to left
\node at (.1, -.2) {$0$};
\node at (-1.9, -.2) {$\eta_4$};
\node at (-3.9, -.2) {$\eta_3$};
\node at (-5.9, -.2) {$\eta_2$};
\node at (-7.9, -.2) {$\eta_1$};
%top row, right to left
\node at (-.9, 1.4) {$\bar{\lambda}_5$};
\node at (-2.9, 1.4) {$\bar{\lambda}_4$};
\node at (-4.9, 1.4) {$\bar{\lambda}_3$};
\node at (-6.9, 1.4) {$\bar{\lambda}_2$};
\node at (-8.9, 1.4) {$\bar{\lambda}_1$};
%formula in figure
\node at (1.9,.7) {$\sum_i \bar{\lambda}_i - \sum_{j} \eta_j = r$};
\end{tikzpicture}
\caption{An interlacing pattern corresponding to the space $[V(\lambda^*) \otimes V(r\mathbf{\omega}_1)\otimes V(\eta)]^{\mathrm{SL}_5(\C)}$.}
\label{pieri2}
\end{figure}
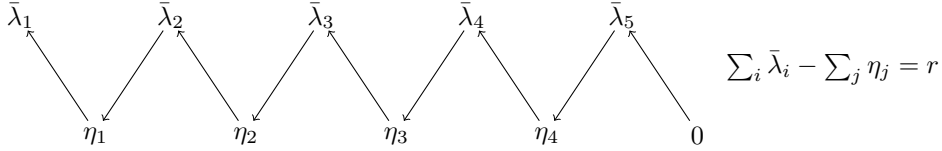

\noindent
Interlacing patterns describing the spaces  $[V(\lambda^*) \otimes V(s\mathbf{\omega}_{m-1}) \otimes V( \eta)]^{\mathrm{SL}_m(\C)}$ are depicted with opposite orientation as in Figure \ref{pieri2dual}.

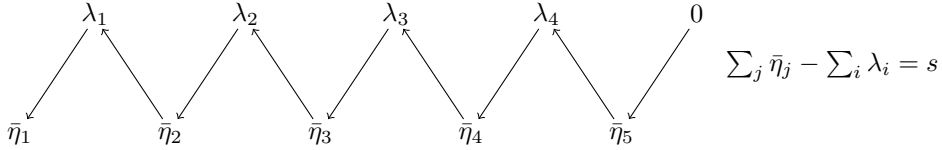
\begin{figure}[htbp]
\begin{tikzpicture}
%The INVERTED interlacing pattern
\draw [->] (0,1.2) -- (-.8, 0);
\draw [->] (-1,0) -- (-1.8, 1.2);
\draw [->] (-2, 1.2) -- (-2.8, 0);
\draw [->] (-3,0) -- (-3.8, 1.2);
\draw [->] (-4, 1.2) -- (-4.8, 0);
\draw [->] (-5,0) -- (-5.8, 1.2);
\draw [->] (-6, 1.2) -- (-6.8, 0);
\draw [->] (-7,0) -- (-7.8, 1.2);
\draw [->] (-8, 1.2) -- (-8.8, 0);
%top row, right to left
\node at (.1, 1.4) {$0$};
\node at (-1.9, 1.4) {$\lambda_4$};
\node at (-3.9, 1.4) {$\lambda_3$};
\node at (-5.9, 1.4) {$\lambda_2$};
\node at (-7.9, 1.4) {$\lambda_1$};
%bottom row, right to left
\node at (-.9, -.2) {$\bar{\eta}_5$};
\node at (-2.9, -.2) {$\bar{\eta}_4$};
\node at (-4.9, -.2) {$\bar{\eta}_3$};
\node at (-6.9, -.2) {$\bar{\eta}_2$};
\node at (-8.9, -.2) {$\bar{\eta}_1$};
%formula in figure
\node at (1.9,.7) {$\sum_j \bar{\eta}_j - \sum_{i} \lambda_i = s$};
\end{tikzpicture}
\caption{An interlacing pattern corresponding to the space $[V(\lambda^*) \otimes V(s\mathbf{\omega}_4) \otimes V(\eta)]^{\mathrm{SL}_5(\C)}$.}
\label{pieri2dual}
\end{figure}

  We define the upper $\del_1(\mathbf{b})$ and lower $\del_2(\mathbf{b})$ boundaries of an interlacing pattern to be the associated $\mathrm{SL}_m(\C)$ weights:

\begin{equation}
\del_1(\mathbf{b}) = \lambda - \lambda_m\mathbf{\omega}_m, \ \ \ \ \del_2(\mathbf{b}) = \eta - \eta_m\mathbf{\omega}_m.\\
\end{equation}

\noindent
For example, the boundary values of the patterns representing the spaces $[V(\lambda^*)\otimes V(r\mathbf{\omega}_1)\otimes V(\eta)]^{\mathrm{SL}_m(\C)}$ and $[V(\lambda^*)\otimes V(s\mathbf{\omega}_{m-1})\otimes V(\eta)]^{\mathrm{SL}_m(\C)}$ are both $(\lambda, \eta).$

 It will also be important to understand the interlacing patterns for spaces of the form $[V(\lambda^*)\otimes V(s_1\mathbf{\omega}_{m-1}) \otimes V(s_2\mathbf{\omega}_{m-1})]^{\mathrm{SL}_m(\C)}$ and $[V(r_1\mathbf{\omega}_1) \otimes V(r_2\mathbf{\omega}_1)\otimes V(\eta)]^{\mathrm{SL}_m(\C)}.$   Using the recipe in Propositions \ref{p2} and \ref{dp}, the interlacing patterns representing these spaces have the first $m-1$ entries of the lower row (respectively upper row) equal.  This in turn forces the first $m-2$ entries of the upper row (respectively lower row) to be equal, see Figures \ref{pieri3} and \ref{pieri4}.

\begin{figure}[htbp]
\begin{tikzpicture}
%The interlacing pattern
\draw [->] (0,0) -- (-.8, 1.2);
\draw [->] (-1,1.2) -- (-1.8, 0);
\draw [->] (-2, 0) -- (-2.8, 1.2);
\draw [->] (-3,1.2) -- (-3.8, 0);
\draw [->] (-4, 0) -- (-4.8, 1.2);
\draw [->] (-5,1.2) -- (-5.8, 0);
\draw [->] (-6, 0) -- (-6.8, 1.2);
\draw [->] (-7,1.2) -- (-7.8, 0);
\draw [->] (-8, 0) -- (-8.8, 1.2);
%bottom row, right to left
\node at (.1, -.2) {$0$};
\node at (-1.9, -.2) {$\eta_4$};
\node at (-3.9, -.2) {$r_1+ \bar{\lambda}_5$};
\node at (-5.9, -.2) {$r_1+ \bar{\lambda}_5$};
\node at (-7.9, -.2) {$r_1+ \bar{\lambda}_5$};
%top row, right to left
\node at (-.9, 1.4) {$\bar{\lambda}_5$};
\node at (-2.9, 1.4) {$r_1+ \bar{\lambda}_5$};
\node at (-4.9, 1.4) {$r_1+ \bar{\lambda}_5$};
\node at (-6.9, 1.4) {$r_1+ \bar{\lambda}_5$};
\node at (-8.9, 1.4) {$r_1+ \bar{\lambda}_5$};
%formula in figure
\node at (1.9,.7) {$r_1 + 2\bar{\lambda}_5 - \eta_4 = r_2$};
\end{tikzpicture}
\caption{An interlacing pattern corresponding to the space $[V(r_1\mathbf{\omega}_1)\otimes V(r_2\mathbf{\omega}_1)\otimes V(\eta)]^{\mathrm{SL}_m(\C)}$.}
\label{pieri3}
\end{figure}
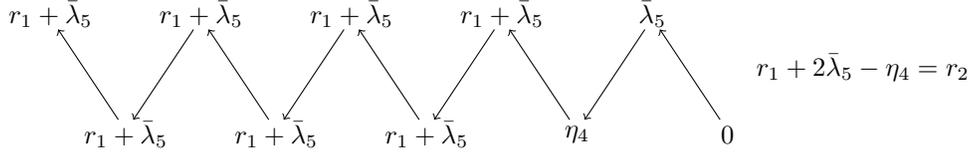

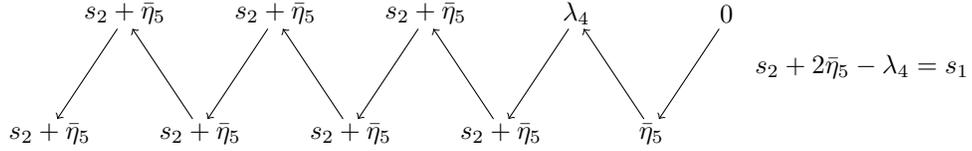
\begin{figure}[htbp]
\begin{tikzpicture}
%The INVERTED interlacing pattern
\draw [->] (0,1.2) -- (-.8, 0);
\draw [->] (-1,0) -- (-1.8, 1.2);
\draw [->] (-2, 1.2) -- (-2.8, 0);
\draw [->] (-3,0) -- (-3.8, 1.2);
\draw [->] (-4, 1.2) -- (-4.8, 0);
\draw [->] (-5,0) -- (-5.8, 1.2);
\draw [->] (-6, 1.2) -- (-6.8, 0);
\draw [->] (-7,0) -- (-7.8, 1.2);
\draw [->] (-8, 1.2) -- (-8.8, 0);
%top row, right to left
\node at (.1, 1.4) {$0$};
\node at (-1.9, 1.4) {$\lambda_4$};
\node at (-3.9, 1.4) {$s_2 + \bar{\eta}_5$};
\node at (-5.9, 1.4) {$s_2 + \bar{\eta}_5$};
\node at (-7.9, 1.4) {$s_2 + \bar{\eta}_5$};
%bottom row, right to left
\node at (-.9, -.2) {$\bar{\eta}_5$};
\node at (-2.9, -.2) {$s_2 + \bar{\eta}_5$};
\node at (-4.9, -.2) {$s_2 + \bar{\eta}_5$};
\node at (-6.9, -.2) {$s_2 + \bar{\eta}_5$};
\node at (-8.9, -.2) {$s_2 + \bar{\eta}_5$};
%formula in figure
\node at (1.9,.7) {$s_2 + 2\bar{\eta}_5 - \lambda_4 = s_1$};
\end{tikzpicture}
\caption{An interlacing pattern corresponding to the space $[V(\lambda^*)\otimes V(s_1\mathbf{\omega}_{m-1})\otimes V(s_2\mathbf{\omega}_{m-1})]^{\mathrm{SL}_m(\C)}$.}
\label{pieri4}
\end{figure}

The conditions of the Pieri rule imply that the entries of an interlacing pattern representing a space of the form $[V(r_1\mathbf{\omega}_1)\otimes V(r_2\mathbf{\omega}_1)\otimes V(\eta)]^{\mathrm{SL}_m(\C)}$ or $[V(\lambda^*)\otimes V(s_1\mathbf{\omega}_{m-1})\otimes V(s_2\mathbf{\omega}_{m-1})]^{\mathrm{SL}_m(\C)}$ are determined by three numbers: $r_1, \eta_{m-1}$, and $r_2$ in the former case and $s_1, \lambda_{m-1},$ and $s_2$ in the latter case.   

We let $L(P_1, P_2, P_3)$ be the set of those $(\lambda_1, \lambda_2, \lambda_3)$ such that $\lambda_i$ is a character of $P_i$ and  $V(\lambda_1) \otimes V(\lambda_2)\otimes V(\lambda_3)$ contains an invariant.  Throughout Subsection \ref{pr} we have seen four distinguished cases: $L(P, P, B)$, $L(B, P, B)$, $L(B,  P^*, B)$, and $L(B, P^*, P^*)$.  These cases orrespond to the spaces of type  $[V(r_1\mathbf{\omega}_1)\otimes V(r_2\mathbf{\omega}_1) \otimes V(\eta)]^{\mathrm{SL}_m(\C)}$, $[V(\lambda)\otimes V(r\mathbf{\omega}_1)\otimes V(\eta)]^{\mathrm{SL}_m(\C)},$ $[V(\lambda)\otimes V(s\mathbf{\omega}_{m-1})\otimes V(\eta)]^{\mathrm{SL}_m(\C)},$ and $[V(s_1\mathbf{\omega}_{m-1})\otimes V(s_2\mathbf{\omega}_{m-1})\otimes V(\eta)]^{\mathrm{SL}_m(\C)}$, respectively.

\begin{lemma}\label{classicalaffinesemigroup}
Each set $L(P, P, B)$, $L(B, P, B)$, $L(B, P^*, B)$, and $L(B, P^*, P^*)$ is a freely generated affine semigroup on $3, 2m-1, 2m-1$ and $3$ generators, respectively.  
\end{lemma}

\begin{proof}
Each of these sets is a subset of $\Lambda_m^3$, and Propositions \ref{p2} and \ref{dp} imply that they are closed under addition. Furthermore, for any interlacing pattern $\mathbf{b}$, $\del_1(\mathbf{b})$ and $\del_2(\mathbf{b})$ satisfy the conditions of Proposition \ref{p2} with $r$ calculated by condition $(2)$.  Condition $(3)$ of Proposition \ref{p2} then implies that $L(B, P, B)$ is isomorphic to $\Lambda_{2m-1}$.  An identical argument proves that $L(B, P^*, B) \cong \Lambda_{2m-1}$, $L(P, P, B) \cong L(B, P^*, P^*) \cong \Lambda_3.$  The proposition then follows from Example \ref{exampleinterlace}. 
\end{proof}

In what follows we refer to the $L(P_1, P_2, P_3)$ appearing in Lemma \ref{classicalaffinesemigroup} as the Pieri semigroups.

\subsection{The algebra of $\mathrm{SL}_m(\C)$ invariant tensors}

Now we bring in a graded algebra $R_n(\mathrm{SL}_m(\C))$ whose graded components are the invariant vector spaces in the $n$-fold tensor products of irreducible $\mathrm{SL}_m(\C)$ representations. Let $A_{\mathrm{SL}_m(\C)}$ be the coordinate ring of the variety $\mathrm{SL}_m(\C)/U$; as a vector space, this algebra is a multiplicity-free direct sum of all irreducible representations of $\mathrm{SL}_m(\C)$, see \cite[Chapter 3]{Gr}: 

\begin{equation}
A_{\mathrm{SL}_m(\C)} = \bigoplus_{\lambda \in \Lambda}V(\lambda).\\
\end{equation}

\noindent
The product operation in this algebra is Cartan multiplication, computed by projecting a tensor product of irreducible representations onto the highest weight component of its direct sum decomposition: 

\begin{equation}
V(\lambda) \otimes V(\eta) \to V(\lambda + \eta).\\
\end{equation}

\noindent
In particular $A_{\mathrm{SL}_m(\C)}$ is graded by the semigroup $\Lambda$, this is equivalent to a right hand side action by the maximal torus $T$ with isotypical components given by the irreducible summands $V(\lambda)$.

 We define $R_n(\mathrm{SL}_m(\C))$ to be the algebra of $\mathrm{SL}_m(\C)$ invariants with respect to the diagonal action on the tensor product $A_{\mathrm{SL}_m(\C)}^{\otimes n}.$  As a vector space, $R_n(\mathrm{SL}_m(\C))$ is a direct sum of all invariant spaces $[V(\vec{\lambda})]^{\mathrm{SL}_m(\C)} = [V(\lambda_1) \otimes \cdots \otimes V(\lambda_n)]^{\mathrm{SL}_m(\C)}$, these summands are the isotypical spaces for the induced action of $T^n$: 

\begin{equation}
R_n(\mathrm{SL}_m(\C)) = \bigoplus_{\vec{\lambda}\in \Lambda^n} (V(\vec{\lambda}))^{\mathrm{SL}_m(\C)}.\\
\end{equation}

The coordinate rings of $\bigwedge^{m-1}(\C^m)$ and $\C^m$  are polynomial algebras on $m$ variables, and as representations, $\C[\bigwedge^{m-1}(\C^m)]$ and $\C[\C^m]$ are the multiplicity free direct sums of the representations in the faces of the Weyl chamber corresponding to $P$ and $P^*$, respectively. Both algebras are graded subalgebras of $A_{\mathrm{SL}_m(\C)}:$

\begin{equation}
\C\Big[\bigwedge^{m-1}(\C^m)\Big] = \bigoplus_{r \geq 0} V(r\mathbf{\omega}_1),  \ \ \ \ \C[\C^m] = \bigoplus_{s \geq 0} V(s\mathbf{\omega}_{m-1}).\\
\end{equation}

\noindent
From this point of view, $R(a, b)$ is a multigraded subalgebra of $R_{a+b}(\mathrm{SL}_m(\C))$.   Now we define several more subalgebras of $R_n(\mathrm{SL}_m(\C))$ along these lines.

\begin{definition}
Let $P_1, \ldots, P_n \subset \mathrm{SL}_m(\C)$ be an $n$-tuple of parabolic subgroups, with associated faces $F_1, \ldots, F_n \subset \Delta$ and subsemigroups $\Lambda_1, \ldots, \Lambda_n \subset \Lambda.$   These define a $T^n$-subalgebra of $R_n(\mathrm{SL}_m(\C))$ as follows:

\begin{equation}
R(\vec{P}) = \bigoplus_{\lambda_1 \in \Lambda_1, \ldots, \lambda_n \in \Lambda_n} \Big(V(\vec{\lambda}))^{\mathrm{SL}_m(\C)} \subset R_n(\mathrm{SL}_m(\C)\Big).\\
\end{equation}

\end{definition}

According to this definition, the algebra $R(a, b)$ is $R(P, \ldots, P, P^*, \ldots, P^*)$. We single out four more algebras of this kind which will be important in what follows: $R(B, P, B)$, $R(B, P^*, B)$, $R(P, P, B)$, and $R(P^*, P^*, B),$ these are the Pieri algebras. 

\begin{proposition}\label{4algebrasclassical}
Each Pieri algebra $R(P_1, P_2, P_3)$ is isomorphic the corresponding semigroup algebra $\C[L(P_1, P_2, P_3)]$. 
\end{proposition}

\begin{proof}
The algebra $R(B, P, B)$ is graded by an action of $T^3$, and the isotypical spaces of this action are precisely the invariant spaces corresponding to the triples of dominant weights appearing in $L(B, P, B).$  Furthermore, each of these isotypical spaces is multiplicity-free by Proposition \ref{p2}, so it follows that $R(B, P, B)$ is isomorphic to the affine semigroup algebra $\C[L(B, P, B)].$  Lemma \ref{classicalaffinesemigroup} implies that $R(B, P, B)$ is a polynomial ring on $2m-1$ variables.  The other three cases are identical. 
\end{proof}

\section{Conformal blocks and the $K$-Pieri algebras}

  We review the construction of the spaces of conformal blocks and the relationship between these spaces and the tensor product invariant spaces of $\mathrm{SL}_m(\C)$.  These two points allow us to bring in the $K$-Pieri rule and prove that each of the $K$-Pieri algebras is a polynomial ring (Proposition \ref{kpierisemigroup}).  We refer the reader to the papers of Beauville \cite{B1}, Looijenga \cite{L} for more on this topic.

\subsection{Conformal blocks and the $K$-Pieri rule}

For $\lambda \in \Lambda$ and $K \in \Z_{\geq 0}$ we let $H(\lambda, K)$ be the integrable highest weight module of the affine Kac-Moody algebra $\hat{sl}_m(\C)$ (\cite[Chapter 12]{Kac}).  We must have $\lambda(h_{\alpha_{1m}}) = \lambda_1 \leq K$ for $h_{\alpha_{1m}}$ the Cartan element associated to the longest postive root of the Lie algebra $sl_m(\C)$. The set $\Delta(K) = \{ \lambda \in \Delta \ | \ \lambda_1 \leq K\}$ is called the level $K$ alcove of the Weyl chamber $\Delta$ and weights $\lambda \in \Lambda(K) = \Lambda \cap \Delta(K)$ are said to be of level $K$.  It is straightforward to check that $\Delta(K)$ and $\Lambda(K)$ are stable under duality $d: \Delta \to \Delta.$

To every marked projective line $(\mathbb{P}^1, \vec{p})$ and assignment of level $K$ weights $p_i \to \lambda_i \in \Delta(K)$ there is a finite dimensional vector space of conformal blocks $V_{\mathbb{P}^1, \vec{p}}(\vec{\lambda}, K).$  This space is constructed as the invariant subspace of the tensor product $H(\vec{\lambda}, K) = H(\lambda_1, K) \otimes \cdots \otimes H(\lambda_n, K)$  with respect to an action by the Lie algebra $\C[\mathbb{P}^1 \setminus \{p_1, \ldots, p_n\}]\otimes sl_m(\C)$ (for a description of this action see \cite[Proposition 2.3]{B1}).   As the markings $\vec{p}$ vary, the spaces $V_{\mathbb{P}^1, \vec{p}}(\vec{\lambda}, K)$ sweep out a vector bundle on the moduli space $\mathcal{M}_{0, n}$, in particular the dimension of $V_{\mathbb{P}^1, \vec{p}}(\vec{\lambda}, K)$ depends only on the marking and level data. There is only one $3$-marked projective line $(\mathbb{P}^1, 0, 1, \infty),$ so in this case the space of conformal blocks $V_{\mathbb{P}^1, 0, 1, \infty}(\lambda, \eta, \mu, K)$ is unique.  From now on we drop the curve $(\mathbb{P}^1, 0, 1, \infty)$ from the notation and simply write $V_{0, 3}(\lambda, \eta, \mu, K).$  The following proposition is the $K$-Pieri rule, which computes the space  $V_{0, 3}(\lambda, \eta, \mu, K)$ when one of the weights is a multiple of $\mathbf{\omega}_1.$   For more background on this rule see \cite{MS}, we will also give a proof later in this section.

\begin{proposition}\label{kpieri}
For $\lambda, \eta \in \Lambda_K$ and $r \leq K$, the space $V_{0,3}(\lambda^*, r\mathbf{\omega}_1, \eta, K)$ is non-trivial if and only if there is a dominant $\mathrm{GL}_m(\C)$ weight $\bar{\lambda} \in \Delta_{\mathrm{GL}_m(\C)}$ such that the following hold:

\begin{enumerate}
\item $\lambda = \bar{\lambda} - \bar{\lambda}_m\mathbf{\omega}_m$,\\
\item $\sum_j \bar{\lambda}_j - \sum_{i} \eta_i = r$,\\
\item $\bar{\lambda}_1 \leq K$,\\ 
\item $\eta \prec \bar{\lambda}$.\\
\end{enumerate}

\noindent
In the case that $(1)-(4)$ hold, $V_{0, 3}(\lambda, r\mathbf{\omega}_1, \eta, K) \cong \C.$
\end{proposition}

\noindent
The entries of the dominant weights $\eta, \bar{\lambda}$ from Proposition \ref{kpieri} can be placed in an interlacing diagram, shown in Figure \ref{pieri} for $m = 5.$

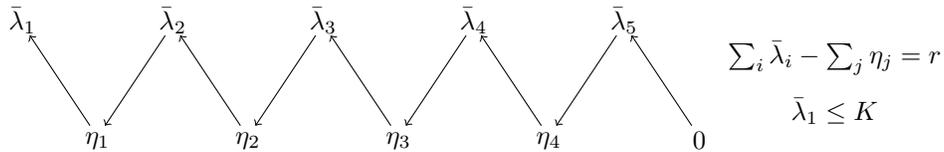
\begin{figure}[htbp]
\begin{tikzpicture}
%The interlacing pattern
\draw [->] (0,0) -- (-.8, 1.2);
\draw [->] (-1,1.2) -- (-1.8, 0);
\draw [->] (-2, 0) -- (-2.8, 1.2);
\draw [->] (-3,1.2) -- (-3.8, 0);
\draw [->] (-4, 0) -- (-4.8, 1.2);
\draw [->] (-5,1.2) -- (-5.8, 0);
\draw [->] (-6, 0) -- (-6.8, 1.2);
\draw [->] (-7,1.2) -- (-7.8, 0);
\draw [->] (-8, 0) -- (-8.8, 1.2);
%bottom row, right to left
\node at (.1, -.2) {$0$};
\node at (-1.9, -.2) {$\eta_4$};
\node at (-3.9, -.2) {$\eta_3$};
\node at (-5.9, -.2) {$\eta_2$};
\node at (-7.9, -.2) {$\eta_1$};
%top row, right to left
\node at (-.9, 1.4) {$\bar{\lambda}_5$};
\node at (-2.9, 1.4) {$\bar{\lambda}_4$};
\node at (-4.9, 1.4) {$\bar{\lambda}_3$};
\node at (-6.9, 1.4) {$\bar{\lambda}_2$};
\node at (-8.9, 1.4) {$\bar{\lambda}_1$};
%formula in figure
\node at (1.9,.9) {$\sum_i \bar{\lambda}_i - \sum_{j} \eta_j = r$};
\node at (1.9, .2){$\bar{\lambda}_1 \leq K$};
\end{tikzpicture}
\caption{An interlacing diagram corresponding to the space $V_{0, 3}(\lambda^*, r\mathbf{\omega}_1, \eta, K)$}
\label{pieri}
\end{figure}

As in the classical case, we define the upper and lower boundaries of a pattern $\mathbf{b}$ representing $V_{0, 3}(\lambda^*, r\mathbf{\omega}_1, \eta, K)$:

\begin{equation}
\del_1(\mathbf{b}) = (\eta, K),  \ \ \ \ \del_2(\mathbf{b}) = (\lambda, K).\\
\end{equation} 

The assumption that $\lambda, \eta \in \Delta(K)$ it must be the case that $\lambda_1, \eta_1 \leq K$, so the boundary maps $\del_1, \del_2$ take values in the semigroup $\Lambda_m$ from Example \ref{exampleinterlace}.  Beauville \cite[Proposition 2.8]{B1} shows that the dimension of a space of conformal blocks does not change if the weights are replaced by their duals, as a consequence there is also a dual $K$-Pieri rule. 

\begin{proposition}\label{dualkpieri}
For $\lambda, \eta \in \Lambda_K$ and $s \leq K$, the space $V_{0,3}(\lambda^*, s\mathbf{\omega}_1, \eta, K)$ is non-trivial if and only if there is a dominant $\mathrm{GL}_m(\C)$ weight $\bar{\eta} \in \Delta_{\mathrm{GL}_m(\C)}$ such that the following hold:

\begin{enumerate}
\item $\bar{\eta} = \eta + \bar{\eta}_m \mathbf{\omega}_m$,\\
\item $\sum_j \bar{\eta}_j - \sum_i \lambda_i^* = s$,\\
\item $\bar{\eta}_1 \leq K$,\\ 
\item $\lambda \prec \bar{\eta}$.\\
\end{enumerate}

\noindent
In the case that $(1)-(4)$ are satisfied, $V_{0, 3}(\lambda^*, s\mathbf{\omega}_1, \eta, K) \cong \C.$
\end{proposition}

\noindent
Interlacing patterns $\bold{b}$ for the spaces $V_{0, 3}(\lambda^*, s\mathbf{\omega}_{m-1}, \eta)$ are depicted with the opposite orientation in Figure \ref{pieridual}.  Boundaries of the dual patterns are defined as expected: $\del_1(\mathbf{b}) = (\eta, K)$, $\del_2(\mathbf{b}) = (\lambda, K).$ In sympathy with the definition of the Pieri semigroups $L(P_1, P_2, P_3)$ we let $Q(P_1, P_2, P_3)$ be the set of $(\lambda_1, \lambda_2, \lambda_3, K$ such that the space $V_{0, 3}(\lambda, \mu, \eta, K) \neq 0$ and $\lambda_i \in P_i.$

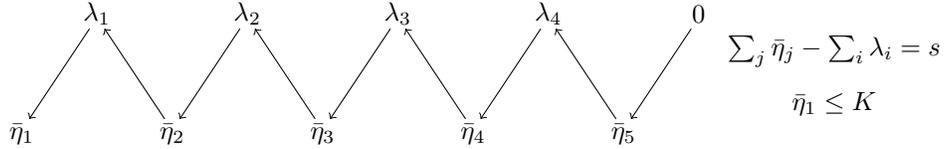
\begin{figure}[htbp]
\begin{tikzpicture}
%The INVERTED interlacing pattern
\draw [->] (0,1.2) -- (-.8, 0);
\draw [->] (-1,0) -- (-1.8, 1.2);
\draw [->] (-2, 1.2) -- (-2.8, 0);
\draw [->] (-3,0) -- (-3.8, 1.2);
\draw [->] (-4, 1.2) -- (-4.8, 0);
\draw [->] (-5,0) -- (-5.8, 1.2);
\draw [->] (-6, 1.2) -- (-6.8, 0);
\draw [->] (-7,0) -- (-7.8, 1.2);
\draw [->] (-8, 1.2) -- (-8.8, 0);
%top row, right to left
\node at (.1, 1.4) {$0$};
\node at (-1.9, 1.4) {$\lambda_4$};
\node at (-3.9, 1.4) {$\lambda_3$};
\node at (-5.9, 1.4) {$\lambda_2$};
\node at (-7.9, 1.4) {$\lambda_1$};
%bottom row, right to left
\node at (-.9, -.2) {$\bar{\eta}_5$};
\node at (-2.9, -.2) {$\bar{\eta}_4$};
\node at (-4.9, -.2) {$\bar{\eta}_3$};
\node at (-6.9, -.2) {$\bar{\eta}_2$};
\node at (-8.9, -.2) {$\bar{\eta}_1$};
%formula in figure
\node at (1.9,.9) {$\sum_j \bar{\eta}_j - \sum_{i} \lambda_i = s$};
\node at (1.9, .2) {$\bar{\eta}_1 \leq K$};
\end{tikzpicture}
\caption{An interlacing diagram corresponding to the space $V_{0, 3}(\lambda^*, s\mathbf{\omega}_{m-1}, \eta, K)$}
\label{pieridual}
\end{figure}

\begin{lemma}\label{kpierisemigroup}
The sets  $Q(P, P, B)$, $Q(B, P^*, B)$, $Q(B, P, B)$, and $Q(B, P^*, P^*)$ are freely generated affine semigroups on  $4$, $2m$, $2m$, and $4$ generators, respectively. 
\end{lemma}

\begin{proof}
It is straightforward to check that $Q(B, P^*, B)$ is closed under addition, so it remains only to check that it is freely generated. Following the definition, $(\lambda, r, \eta, K) \in Q(B, P^*, B)$ if and only if $K \geq \bar{\lambda}_1 \geq \eta_1 \geq \cdots \geq \bar{\lambda}_{m-1} \geq \eta_{m-1} \geq \bar{\lambda}_m$ (recall that $\bar{\lambda}$ is recoverable from $\lambda, \eta, r$).  It follows from Example \ref{exampleinterlace} that this semigroup is freely generated.  An identical proof applies to the other three cases.     
\end{proof}

\noindent
From now on we refer to the sets in Lemma \ref{kpierisemigroup} are the $K$-Pieri semigroups.

\subsection{Conformal blocks, principal bundles, and correlation}

Several authors (see e.g. \cite{P}, \cite{KNR}, \cite{BL}, \cite{LS}, \cite{BLS}, \cite{Fal}) have shown that the conformal block vector spaces coincide with the global sections of line bundles on the moduli stack $\mathcal{M}_{\mathbb{P}^1, \vec{p}}(\vec{P})$:

\begin{equation}
H^0(\mathcal{M}_{\mathbb{P}^1, \vec{p}}(\vec{P}), \mathcal{L}(\vec{\lambda}, K)) = V_{\mathbb{P}^1, \vec{p}}(\vec{\lambda}, K).\\
\end{equation}

\noindent
As a consequence, the total coordinate ring $V_{\mathbb{P}^1, \vec{p}}(\vec{P})$ is a graded sum of all the spaces of conformal blocks on $(\mathbb{P}^1, \vec{p})$:

\begin{equation}
V_{\mathbb{P}^1, \vec{p}}(\vec{P}) = \bigoplus_{\lambda_1 \in \Lambda_i, \ldots, \lambda_n \in \Lambda_n, K \geq 0} V_{\mathbb{P}^1, \vec{p}}(\vec{\lambda}, K).\\
\end{equation}   

\noindent
Here $\Lambda_i$ is the monoid of dominant weights associated to $P_i$, this condition and the non-negativity of $K$ are necessary for $\mathcal{L}(\vec{\lambda}, K)$ to be effective. This direct sum expression is the isotypical decomposition of $V_{\mathbb{P}^1, \vec{p}}(\vec{P})$ with respect to a $T^n\times \C^*$ action, where the $\C^*$ action is derived from the grading by the level.  

 As the $K$-Pieri rule indicates, the representation theory of $\mathrm{SL}_m(\C)$ and the theory of $\hat{sl}_m(\C)$ conformal blocks are closely related.  This is  succinctly captured in the following theorem; the vector space version of this statement is \cite[Proposition 4.1]{B1}, we present the conclusion of \cite[Section 2]{M4}.  

\begin{theorem}\label{algebracorrelation}
Give the polynomial algebra $\C[t]$ its $\C^*$ action. There is a $T^n \times \C^*$-equivariant inclusion of algebras:

\begin{equation}
\Phi_{\mathbb{P}^1, \vec{p}}: V_{\mathbb{P}^1, \vec{p}}(\vec{P}) \to R(\vec{P})\otimes \C[t],\\
\end{equation}

\begin{equation}
\Phi_{\mathbb{P}^1, \vec{p}}: V_{\mathbb{P}^1, \vec{p}}(\vec{\lambda}, K) \to (V(\vec{\lambda}))^{\mathrm{SL}_m(\C)}t^K.\\
\end{equation}

\end{theorem}

In fact, by \cite[Proposition 2]{M10} the algebra $V_{\mathbb{P}^1, \vec{p}}(\vec{P})$ is a Rees algebra of $R(\vec{P})$ with respect to a discrete valuation on $R(\vec{P})$ (this proposition is written in the $\mathrm{SL}_3(\C)$ case, but applies to general simple, simply connected groups).  In particular, for every marked projective line $(\mathbb{P}^1, \vec{p})$ there is a discrete valuation over $\C$: $v_{\vec{p}}: R(\vec{P}) \to \Z \cup \{-\infty\}$ which we call the threshold valuation (see \cite[Section 4]{M10}), such that the following holds:

\begin{equation}
V_{\mathbb{P}^1, \vec{p}}(\vec{\lambda}, K) = \{f \ | \ v_{\vec{p}}(f) \leq K\} \subset (V(\vec{\lambda}))^{\mathrm{SL}_m(\C)}.\\
\end{equation}

Recall that a discrete valuation is subadditive: $v_{\vec{p}}(f + g) \leq MAX\{v_{\vec{p}}(f), v_{\vec{p}}(g)\}$, multiplicative: $v_{\vec{p}}(fg) = v_{\vec{p}}(f) + v_{\vec{p}}(g)$, sends non-zero scalars to $0$: $v_{\vec{p}}(C) = 0$, and sends $0$ to $-\infty$: $v_{\vec{p}}(0) = -\infty.$  Also, any non-scalar invariant must have threshhold value at least $1$, so this function only sends scalars to $0$. These properties will be useful in determining the structure of  $V_{\mathbb{P}^1, \vec{p}}(\vec{P})$ from that of $R(\vec{P})$ in the cases we will consider.  Five special cases are relevant: $V_{\mathbb{P}^1, \vec{p}}(a, b)$ is a Rees algebra of $R(a, b)$, and each $K$-Pieri algebra $V_{0, 3}(P_1, P_2, P_3)$ is a Rees algebra of the corresponding Pieri algebra $R(P_1, P_2, P_3).$

\begin{proposition}\label{quantumaffinesemigroup}
Each of the $K$-Pieri algebras is an affine semigroup algebra. 
\end{proposition}

\begin{proof}
We make our argument for $V_{0, 3}(B, P, B)$, the other three cases are identical.  By Proposition \ref{4algebrasclassical}, $R(B, P, B) \cong \C[L(B, P, B)]$, and by Theorem \ref{algebracorrelation}, $V_{0, 3}(B, P, B)$ is a graded subalgebra of the affine semigroup algebra $\C[L(B, P, B)\times \Z_{\geq 0}].$  It follows that each graded component $V_{0, 3}(\lambda, r\mathbf{\omega}_1, \eta, K)$ can be identified with an interlacing pattern and a non-negative integer $(\mathbf{b}, K)$, and that multiplication of components corresponds to the expected addition operation on these pairs. 
\end{proof}

We let $Q'(P_1, P_2, P_3)$ be the subsemigroup of $L(P_1, P_2, P_3)\times \Z_{\geq 0}$ defined in the conclusion of Proposition \ref{quantumaffinesemigroup}.  In what follows we prove each of these semigroups are isomorphic to the corresponding $K$-Pieri semigroup, establishing the $K$-Pieri rule and that the $K$-Pieri algebras are polynomial rings.  To do this we need a characterization of the space $V_{0, 3}(\lambda, \mu, \eta, K)$ found in \cite{U}.  Consider the copy of $\mathrm{SL}_2(\C)$ inside $\mathrm{SL}_m(\C)$ which corresponds to the longest root $\alpha_{1m}.$ An irreducible representation $V(\lambda)$ can be decomposed into isotypical components of this subgroup $V(\lambda) = \bigoplus_{i \geq 0} W_{\lambda, i},$ this decomposition is used to define the following subspace $W_K \subset V(\lambda) \otimes V(\eta) \otimes V(\mu)$: 

\begin{equation}
W_K = \bigoplus_{i + j + k \leq 2K} W_{\lambda, i} \otimes W_{\eta, j} \otimes W_{\mu, k}.\\
\end{equation}

\noindent
The following is \cite[3.5.2]{U} and \cite[Proposition 4.3]{B1}. 

\begin{proposition}\label{ueno}
The space of conformal blocks $V_{0, 3}(\lambda, \mu, \eta, K)$ can be identified with the space $W_K \cap (V(\lambda)\otimes V(\mu) \otimes V(\eta))^{\mathrm{SL}_m(\C)}.$
\end{proposition}

Proposition \ref{ueno} can be used to explicitely compute the valuation $v_{0, 3}$ on the Pieri algebras, because it allows us to find the threshhold levels of the generators of these algebras.    The generators $\mathbf{w}_{s, t}$ of the Pieri semigroups each represent the unique invariant in a tensor product of the form $\bigwedge^k(\C^m) \otimes \C^m \otimes \bigwedge^j(\C^m)$ and $\bigwedge^k(\C^m) \otimes \bigwedge^{m-1}(\C^m) \otimes \bigwedge^j(\C^m),$ where $k + j + 1 \in m\Z,$ and $\bigwedge^k(\C^m) \otimes \bigwedge^j(\C^m),$ where $k + j \in m\Z.$  These invariants have the following form as tensors: 

\begin{equation}
T_{i, 1, j} = \sum_{|I| =i, |J| = j} (-1)^{\sigma(I, k, J)} z_I \otimes z_k \otimes z_J,\\
\end{equation}

$$T_{i, m-1, j} = \sum_{|I| = i, |K| = m-1 ,|J| = j} (-1)^{\sigma(I, K, J)} z_I \otimes z_K \otimes z_J,$$

$$P_{i, j} = \sum_{|I| = i, |J| = j} (-1)^{\sigma(I, J)} z_I \otimes z_J.$$

\noindent
Here $z_I$ is a wedge product of basis vectors over the indicated index set, and $(-1)^{\sigma(I, K, J)}$ is a sign function, the following is \cite[Proposition 3.4]{KM}.

\begin{lemma}\label{genlevel}
Each generating invariant of the algebras $R(B, P, B)$, $R(B, P^*, B)$, $R(P, P, B)$ and $R(P^*, P^*, B)$ has threshhold value equal to $1.$
\end{lemma}

\begin{proposition}\label{4algebrasquantum}
Each $K$-Pieri semigroup $Q(P_1, P_2, P_3)$ is isomorphic to its counterpart $Q'(P_1, P_2, P_3).$
\end{proposition}

\begin{proof}
We will prove $Q(B, P^*, B) \cong Q'(B, P^*, B)$, the other three cases are similar. Note that both of these semigroups are subsets of $L(B, P^*, B)\times \Z_{\geq 0}.$ 
First we observe that each generator $\mathbf{w}_{ij} \in L(B, P^*, B)$ has first top row entry less than or equal to $1$, which is the threshhold value of the corresponding generator of $R(B, P^*, B)$.  It follows that $(\mathbf{w}_{ij}, 1) \in Q'(B, P^*, B)$.  Furthermore, these elements generate $Q(B, P^*, B)$, along with $(\bold{0}, 1)$ by Lemma \ref{kpierisemigroup}. If $(\mathbf{b}, K) \in Q'(B, P^*, B)$, then $\mathbf{b}$ can be represented uniquely as a sum $\sum K_i\mathbf{w}_{ij}.$  It follows that the corresponding element in $R(B, P^*, B)$ has thresshold value equal to $\sum K_i \leq K$, and that we can write $(\mathbf{b}, K) = \sum K_i(\mathbf{w}_{ij}, 1) + (K- \sum K_i) (\bold{0}, 1).$ This proves the proposition. 
\end{proof}

\section{Structure of $R_{\tree_0}(a, b)$ and $V_{\tree_0}(a, b)$}\label{abdegen}

Now that we've shown that the Pieri algebras and the $K$-Pieri algebras are polynomial rings, we can use Proposition \ref{kos} to find presentations of the algebras $R_{\tree_0}(a, b)$ and $V_{\tree_0}(a, b)$.  We review the construction of these algebras, their relationships to $R(a, b)$ and $V_{\mathbb{P}^1, \vec{p}}(a, b)$, and then we prove the generation and Koszul statements in Theorem \ref{main}.

Figure \ref{cat} depicts the tree $\tree_0$, which is the combinatorial backbone
to the construction of $R_{\tree_0}(a, b)$ and $V_{\tree_0}(a, b)$.   This figure also shows the forest $\hat{\tree}_0$, which is obtained from $\tree_0$ by splitting each edge which is not attached to a leaf. 

\begin{figure}[htbp]
\centering
\includegraphics[scale = 0.5]{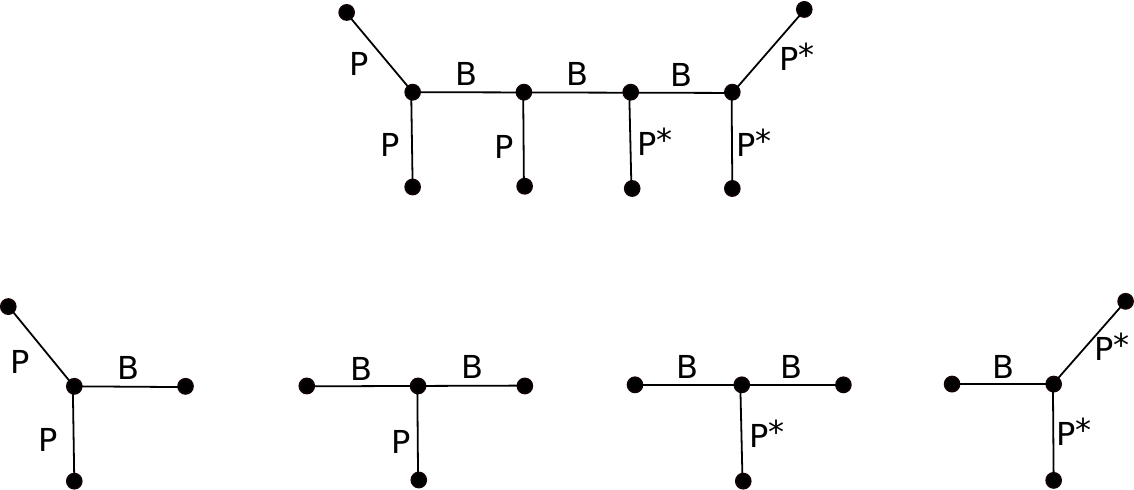}
\caption{The tree $\tree_0$ and the forest $\hat{\tree}_0$.}
\label{cat}
\end{figure}

 We label the following elements of $\tree_0$ from left to right: the leaves are $\ell_1, \ldots, \ell_n$, the non-leaf vertices are $v_1, \ldots, v_{n-2}$, the non-leaf edges (those labelled by $B$ in Figure \ref{cat}) are $h_1, \ldots, h_{n-3}$, and the leaf edges (those not labelled by $B$) are $g_1, \ldots, g_n$.  We will use the same labels for the corresponding elements in $\hat{\tree}_0$, with the exception of the $h_i$ which are replaced with new split edges $e_i$ and $f_i.$  The vertex $v_i$ in $\hat{\tree}_0$ $2 \leq i \leq n-3$ is in the boundary of $e_i$, $f_{i+1}$, and $g_{i+1}$, and $v_1, v_{n-2}$ are in the boundary of $g_1, g_2, f_1$ and $e_{n-3}, g_{n-1}, g_n$, respectively.  

To each of the non-leaf vertices $v_i$ of $\hat{\tree}_0$ we can associate an algebra $V_{0, 3}(P_1, P_2, P_3)$, where the $P_k$ are chosen in accordance with the labels.   Each of these algebras has an action of $T \times T \times T$, and each $T$ in this product is naturally associated to one of the edges of the corresponding vertex. In this way, the product torus $(T\times T \times T)^{n-2}$ acts on the tensor product $V_{0, 3}(P, P, B)\otimes \cdots \otimes V(B, P^*, P^*)$.  Let $T\times T$ be the subproduct associated to the edges $e_i, f_i$, this pair acts on the direct summand $V_{0, 3}(r_1\mathbf{\omega}_1, r_2\mathbf{\omega}_1, \eta_1, K_1)\otimes$ $\cdots$ $\otimes V_{0, 3}(\lambda_{n-3}, s_{n-1}\mathbf{\omega}_{m-1}, s_n \mathbf{\omega}_{m-1}, K_{n-2})$ with the character $(\eta_i, \lambda_i)$.  We let $T_i \subset T\times T \times \C^* \times \C^*$ be the subtorus (isomorphic to $T \times \C^*$) which acts on this same summand with character $(\eta_i - \lambda_i^*, K_{i-1}-K_i)$. The algebra $V_{\tree_0}(a, b)$ is defined as the invariant subalgebra with respect to the action of the torus  $\mathbb{T} = \prod_{i =1}^{n-2} T_i$:

\begin{equation}
V_{\tree_0}(a, b) = [V_{0, 3}(P, P, B) \otimes \cdots \otimes V_{0, 3}(B, P^*, P^*)]^{\mathbb{T}}.\\
\end{equation}

 We assign the $K$-Pieri semigroups to the vertices of $\hat{\tree}_0$ in the manner above, letting $Q_i$ denote the semigroup at vertex $v_i.$ For every pair $e_i, f_i$ we consider the corresponding boundary maps $\del_2, \del_1: Q_{i-1}, Q_i \to \Lambda_m.$  For an interlacing pattern $\mathbf{b}$ reprsenting the space $V_{0, 3}(\lambda, r\mathbf{\omega}_1, \eta, K)$, the boundary maps $\del_1, \del_2$ return $(\lambda^*, K)$ and $(\eta, K)$, respectively, in particular the boundary maps take values in the (graded) freely generated semigroup $\Lambda_m$.  Accordingly, we define $Q(a, b)$ to be the toric fiber product semigroup $Q_1\times_{\Lambda_m}\times \cdots \times_{\Lambda_m}Q_{n-2}.$

\begin{lemma}\label{Ksemiiso}
The invariant algebra $V_{\tree_0}(a, b)$ is isomorphic to the affine semigroup algebra $\C[Q(a, b)]$. 
\end{lemma}

\begin{proof}
By Proposition \ref{quantumaffinesemigroup}, the tensor product $V_{0, 3}(P, P, B)\otimes \cdots \otimes V(B, P^*, P^*)$ is isomorphic to $\C[Q_1\times \cdots \times Q_{n-2}],$ so it suffices to check that the toric fiber product $Q(a, b)$ and the action of the torus $\mathbb{T}$ pick out the same subalgebra of this algebra.  A summand $V_{0, 3}(r_1\mathbf{\omega}_1, r_2\mathbf{\omega}_1, \eta_1, K_1)\otimes$ $\cdots$ $\otimes V_{0, 3}(\lambda_{n-3}, s_{n-1}\mathbf{\omega}_{m-1}, s_n \mathbf{\omega}_{m-1}, K_{n-2})$ is in $V_{\tree_0}(a, b)$ if and only if $\eta_i = \lambda_i^*$ and $K_i = K_{i+1},$ but this is precisely
the condition that the $\del_1$ and $\del_2$ values coincide at each pair $e_i, f_i$.
\end{proof}

Each of the generators $\mathbf{w}_{ij}$ of the $K$-Pieri semigroups maps to a generator in $\Lambda_m$ under $\del_1$ and $\del_2$.  These observations, along with Lemma \ref{Ksemiiso}, Proposition \ref{4algebrasquantum}, and Proposition \ref{kos} imply the following theorem. 

\begin{theorem}\label{kpresent}
The algebra $V_{\tree_0}(a, b)$ is generated by its summands with $K =1$, and the ideal of forms vanishing on these generators has a quadratic, square-free Gr\"obner basis.
\end{theorem}

Since $V_{\tree_0}(a, b)$ is a flat degeneration of $V_{\mathbb{P}, \vec{p}}(a, b)$ by \cite[Theorem 1.1]{M4}, the same argument used in \cite[Theorem 1.11]{M6} implies that $V_{\mathbb{P}, \vec{p}}(a, b)$ is generated by the conformal blocks of level $1$ and is a Koszul algebra when $\vec{p}$ is generic.

\subsection{The degeneration $R_{\tree_0}(a, b)$}

Now we apply the construction from the previous subsection to $R(a, b).$  Using the Pieri algebras we define $R_{\tree_0}(a, b)$ using the tree $\tree_0$ as a combinatorial guide, as in the construction of $V_{\tree_0}(a, b).$  We let $L_1, \ldots, L_{n-2}$ be the Pieri semigroups associated to the vertices of $\tree_0$, and we define $L(a, b)$ to be the toric fiber product $L_1\times_{\Lambda_{m-1}} \cdots \times_{\Lambda_{m-1}} L_{n-2}$. 

\begin{proposition}\label{semiiso}
The affine semigroup algebra $\C[L(a, b)]$ is isomorphic to $R_{\tree_0}(a, b)$, and is a flat degeneration of $R(a, b)$.
\end{proposition}

\begin{proof}
This follows from an almost identical argument to Lemma \ref{Ksemiiso}, and the degeneration construction in \cite[Theorem 1.8]{M4}.
\end{proof}

\noindent
The hypotheses of Proposition \ref{kos} are also satisfied by $L_1, \ldots, L_{n-2}$ and the boundary maps $\del_1, \del_2$, so we can draw the same conclusions as in Theorem \ref{kpresent}: $R_{\tree_0}(a, b)$ is generated by the products of the generators of the Pieri algebras, and these are subject to a quadratic Gr\"obner basis.  In the next section we will explore the presentation we derive for $R(a, b)$ from the degeneration to $R_{\tree_0}(a, b)$ and compare it to the one constructed in the first fundamental theorem of invariant theory. 

Propositions \ref{semiiso} and Proposition \ref{Ksemiiso} imply the more elaborate conclusion that $V_{\tree_0}(a, b)$ is a Rees algebra of $R_{\tree_0}(a, b)$.  To prove this, we note that each $\C[Q_i]$ is a Rees algebra of $\C[L_i]$, so that $\C[\prod Q_i]$ is a subalgebra of $\C[\prod (L_i \times \Z_{\geq 0})].$ The $\Z_{\geq 0}$ components of the latter semigroup correspond to the level parameters in the $Q_i$, so we may extend the boundary maps $\del_1, \del_2$ on the $L_i$ to include these components.  The toric fiber product operation has the effect of forcing these new parameters to be equal, which produces an inclusion $\C[Q(a,b)] \subset \C[L(a, b)\times \Z_{\geq 0}].$  On the level of semigroups, this identifies $L(a, b) \times \Z_{\geq 0}$ with the $\Z_{\geq 0}$ parameter $\leq K$ with the elements of $Q(a, b)$ with level $\leq K$.

\section{Interlacing patterns}\label{interlacing}

We can now realize $Q(a, b)$ and $L(a, b)$ as affine semigroups of non-negative integer weightings of an interlacing pattern.  This leads us to a proof that $R_{\tree_0}(a, b)$, $V_{\tree_0}(a, b)$, and therefore $V_{\mathbb{P}, \vec{p}}(a, b)$ are Gorenstein algebras.  We also give explicit presentations of $V_{\tree_0}(a, b)$ and $R_{\tree_0}(a, b)$.  In this discussion we leave out the cases $a = 0$ or $b = 0$, these can be handled by the same methods and we leave the details to the reader. 

\subsection{$Q(a, b)$ and $L(a, b)$ as semigroups of interlacing patterns}

Proposition \ref{kpieri} shows that the $K$-Pieri semigroups can be represented as the non-negative integer weightings of a two row interlacing pattern.  The bottom row of this pattern has $m-1$ entries, and the top row has $m$ entries, plus $1$ if we include the level parameter.   The semigroup $Q(a, b)$ can also be represented as a set of weightings of an interlacing pattern, to show this we will require a more flexible representation of the spaces $V_{0, 3}(\lambda, r\mathbf{\omega}_1, \eta, K)$ and $V_{0, 3}(\lambda, s\mathbf{\omega}_{m-1}, \eta, K).$

\begin{lemma}\label{mmkpieri}
The following data are equivalent:

\begin{enumerate}
\item A pair of interlacing $\mathrm{GL}_m(\C)$ weights $\eta \prec \lambda$, with $K \geq \lambda_1 - \eta_m$ and $\sum \lambda_i - \sum \eta_j = r$.\\
\item $V_{0, 3}((\lambda- \lambda_m\mathbf{\omega}_m)^*, r\mathbf{\omega}_1, (\eta - \eta_m\mathbf{\omega}_m), K) \neq 0,$ and a choice $\eta_m \in \Z$.
\end{enumerate}
\end{lemma}

\begin{proof}
Let $\bold{e}$ be the interlacing pattern of two length $m$ rows of $1$'s.  Given an interlacing pattern $\mathbf{b}$ representing the information in $(2)$, we add $\eta_m\mathbf{e}$ to obtain the information in $(1)$. This argument can be reversed. 
\end{proof}

We illustrate why Lemma \ref{mmkpieri} is necessary with a construction.  Let us fix two non-zero spaces of conformal blocks, $V_{0, 3}(\lambda^*, r_1\mathbf{\omega}_1, \eta, K)$, $V_{0, 3}(\eta^*, r_2\mathbf{\omega}_1, \mu, K)$ and a non-zero integer $\mu_m$, and build an interlacing pattern which represents the $1$-dimensional tensor product space $V_{0, 3}(\lambda^*, r_1\mathbf{\omega}_1, \eta, K)\otimes V_{0, 3}(\eta^*, r_2\mathbf{\omega}_1, \mu, K).$  We use the procedure  in the proof of Lemma \ref{mmkpieri} with input $\bar{\mu}_m$, $\eta$, $r_2$, $\mu$, and $K$ to obtain an interlacing pattern $\mathbf{b}_1$ with upper row $\bar{\eta} = \eta + [\bar{\mu}_m + \frac{1}{m}(r + \sum \mu_i - \sum \eta_j)]\mathbf{\omega}_m$.   Then we can carry out this same procedure with input $\bar{\eta}_m$, $\lambda$, $r_1$, $\eta$, and $K$ to obtain $\mathbf{b}_2$.  Since the lower row of $\mathbf{b}_2$ matches the upper row of $\mathbf{b}_1$, these two patterns can be glued together, this is shown in Figure \ref{cpattern}.   Given such a $3$-row pattern $\mathbf{b}$ with $\bar{\lambda}_1 - \bar{\eta}_m, \bar{\eta}_1 - \bar{\mu}_m \leq K$, we can remove multiples of $\mathbf{e}$ from the first and second, respectively second and third pairs of rows to obtain a pair of interlacing diagrams $\mathbf{b}_1, \mathbf{b}_2$ which satisfy the hypotheses of Proposition \ref{kpieri}.  This construction is easily generalized to $N$-row diagrams.

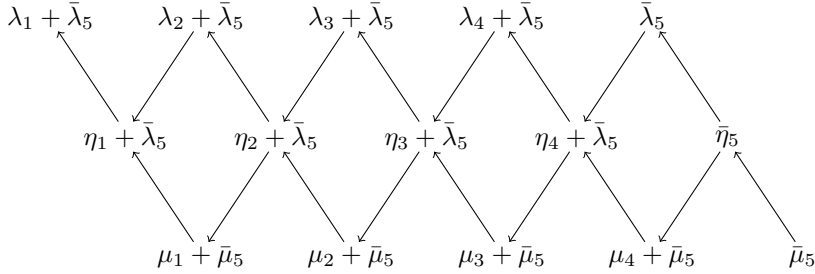
\begin{figure}[htbp]
\begin{tikzpicture}
%The composite interlacing pattern, bottom line of arrows.
\draw [->] (0,0) -- (-.8, 1.2);
\draw [->] (-1,1.2) -- (-1.8, 0);
\draw [->] (-2, 0) -- (-2.8, 1.2);
\draw [->] (-3,1.2) -- (-3.8, 0);
\draw [->] (-4, 0) -- (-4.8, 1.2);
\draw [->] (-5,1.2) -- (-5.8, 0);
\draw [->] (-6, 0) -- (-6.8, 1.2);
\draw [->] (-7,1.2) -- (-7.8, 0);
\draw [->] (-8, 0) -- (-8.8, 1.2);
%The composite interlacing pattern, top line of arrows.
\draw [->] (-1,1.6) -- (-1.8, 2.8);
\draw [->] (-2,2.8) -- (-2.8, 1.6);
\draw [->] (-3,1.6) -- (-3.8, 2.8);
\draw [->] (-4,2.8) -- (-4.8, 1.6);
\draw [->] (-5,1.6) -- (-5.8, 2.8);
\draw [->] (-6,2.8) -- (-6.8, 1.6);
\draw [->] (-7,1.6) -- (-7.8, 2.8);
\draw [->] (-8,2.8) -- (-8.8, 1.6);
\draw [->] (-9,1.6) -- (-9.8, 2.8);
%bottom row, right to left
\node at (.1, -.2) {$\bar{\mu}_5$};
\node at (-1.9, -.2) {$\mu_4 + \bar{\mu}_5$};
\node at (-3.9, -.2) {$\mu_3 + \bar{\mu}_5$};
\node at (-5.9, -.2) {$\mu_2 + \bar{\mu}_5$};
\node at (-7.9, -.2) {$\mu_1 + \bar{\mu}_5$};
%middle row, right to left
\node at (-.9, 1.4) {$\bar{\eta}_5$};
\node at (-2.9, 1.4) {$\eta_4 + \bar{\lambda}_5$};
\node at (-4.9, 1.4) {$\eta_3 + \bar{\lambda}_5$};
\node at (-6.9, 1.4) {$\eta_2 + \bar{\lambda}_5$};
\node at (-8.9, 1.4) {$\eta_1 + \bar{\lambda}_5$};
%top row, right to left
\node at (-1.9, 3) {$\bar{\lambda}_5$};
\node at (-3.9, 3) {$\lambda_4 + \bar{\lambda}_5$};
\node at (-5.9, 3) {$\lambda_3 + \bar{\lambda}_5$};
\node at (-7.9, 3) {$\lambda_2 + \bar{\lambda}_5$};
\node at (-9.9, 3) {$\lambda_1 + \bar{\lambda}_5$};
%formula in figure
\node at (-9, 4) {$\bar{\eta}_5 = \bar{\mu}_5 + \frac{1}{5}(r_2 + \sum \mu_i - \sum \eta_j)$};
\node at (-3, 4){$\bar{\lambda}_5 = \bar{\eta}_5 + \frac{1}{5}(r_1 + \sum \eta_i - \sum \lambda_j)$};
\end{tikzpicture}
\label{cpattern}
\caption{A composite interlacing pattern.}
\end{figure}

\begin{lemma}
The following data are equivalent:

\begin{enumerate}\label{comp1}
\item $N$ interlacing $\mathrm{GL}_m(\C)$ weights $\lambda^1 \prec \cdots \prec \lambda^N$ with $K \geq \lambda^k_1 - \lambda^{k-1}_m$ and $\sum \lambda^k_i - \sum \lambda^{k-1})j = r_{k-1}.$\\
\item $\bigotimes_{k = 1}^{N-1} V_{0, 3}((\lambda^{k+1} - \lambda^{k+1}_m\mathbf{\omega}_m)^*, r_k\mathbf{\omega}_1,\lambda^{k} - \lambda^{k}_m\mathbf{\omega}_m, K) \neq 0$, and a choice $\lambda^1_m \in \Z.$\\ 
\end{enumerate}
\end{lemma}

 Proposition \ref{dualkpieri}, along with the same proof used for Lemma \ref{mmkpieri} give the corresponding construction for dual spaces of conformal blocks.

\begin{lemma}\label{comp2}
The following data are equivalent:

\begin{enumerate}
\item $N$ interlacing $\mathrm{GL}_m(\C)$ weights $\eta^1 \prec \cdots \prec \eta^M$ with $K \geq \eta^k_1 - \eta^{k-1}_m$ and $\sum \eta^k_i - \sum \eta^{k-1})j = s_{k-1}.$\\
\item $\bigotimes_{k = 1}^{M-1} V_{0, 3}((\eta^{k} - \eta^{k}_m\mathbf{\omega}_m)^*, s_k\mathbf{\omega}_{m-1},\eta^{k+1} - \eta^{k+1}_m\mathbf{\omega}_m, K) \neq 0$, and a choice $\eta^1_m \in \Z.$\\ 
\end{enumerate}

\end{lemma}

\noindent
When $\eta^1$ from Lemma \ref{comp2} equals $\lambda^1$ from Lemma \ref{comp1}, we can arrange the weights in a ``wedge'' diagram as in Figure \ref{wedge}.

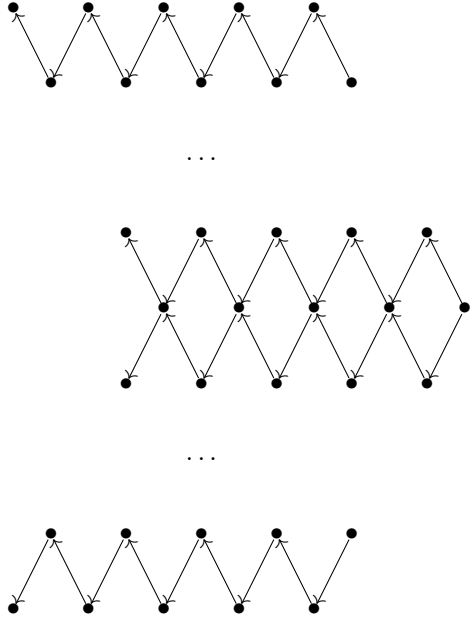
\begin{figure}[htbp]

$$
\begin{xy}
(0, 0)*{\bullet} = "A0";
(10, 0)*{\bullet} = "A1";
(20,0)*{\bullet} = "A2";
(30,0)*{\bullet} = "A3"; 
(40,0)*{\bullet} = "A4"; 
(5,-10)*{\bullet} = "B0";
(15,-10)*{\bullet} = "B1";
(25, -10)*{\bullet} = "B2"; 
(35, -10)*{\bullet} = "B3";
(45, -10)*{\bullet} = "B4";
"A0"; "B0";**\dir{-}? >* \dir{>};
"B0"; "A1";**\dir{-}? >* \dir{>};
"A4"; "B4";**\dir{-}? >* \dir{>};
"B3"; "A4";**\dir{-}? >* \dir{>};
"A3"; "B3";**\dir{-}? >* \dir{>};
"B2"; "A3";**\dir{-}? >* \dir{>};
"A2"; "B2";**\dir{-}? >* \dir{>};
"B1"; "A2";**\dir{-}? >* \dir{>};
"A1"; "B1";**\dir{-}? >* \dir{>};
(25, -20)*{\ldots};
(15, -30)*{\bullet} = "x0";
(25, -30)*{\bullet} = "x1";
(35, -30)*{\bullet} = "x2";
(45, -30)*{\bullet} = "x3"; 
(55, -30)*{\bullet} = "x4"; 
(20,-40)*{\bullet} = "y0";
(30,-40)*{\bullet} = "y1";
(40, -40)*{\bullet} = "y2"; 
(50, -40)*{\bullet} = "y3";
(60, -40)*{\bullet} = "y4";
(15, -50)*{\bullet} = "z0";
(25, -50)*{\bullet} = "z1";
(35, -50)*{\bullet} = "z2";
(45, -50)*{\bullet} = "z3"; 
(55, -50)*{\bullet} = "z4"; 
"x0"; "y0";**\dir{-}? >* \dir{>};
"y0"; "x1";**\dir{-}? >* \dir{>};
"x4"; "y4";**\dir{-}? >* \dir{>};
"y3"; "x4";**\dir{-}? >* \dir{>};
"x3"; "y3";**\dir{-}? >* \dir{>};
"y2"; "x3";**\dir{-}? >* \dir{>};
"x2"; "y2";**\dir{-}? >* \dir{>};
"y1"; "x2";**\dir{-}? >* \dir{>};
"x1"; "y1";**\dir{-}? >* \dir{>};
"z0"; "y0";**\dir{-}? >* \dir{>};
"y0"; "z1";**\dir{-}? >* \dir{>};
"z4"; "y4";**\dir{-}? >* \dir{>};
"y3"; "z4";**\dir{-}? >* \dir{>};
"z3"; "y3";**\dir{-}? >* \dir{>};
"y2"; "z3";**\dir{-}? >* \dir{>};
"z2"; "y2";**\dir{-}? >* \dir{>};
"y1"; "z2";**\dir{-}? >* \dir{>};
"z1"; "y1";**\dir{-}? >* \dir{>};
(25, -60)*{\ldots};
(5,-70)*{\bullet} = "a0";
(15,-70)*{\bullet} = "a1";
(25, -70)*{\bullet} = "a2"; 
(35, -70)*{\bullet} = "a3";
(45, -70)*{\bullet} = "a4";
(0, -80)*{\bullet} = "b0";
(10, -80)*{\bullet} = "b1";
(20,-80)*{\bullet} = "b2";
(30,-80)*{\bullet} = "b3"; 
(40,-80)*{\bullet} = "b4"; 
"b0"; "a0";**\dir{-}? >* \dir{>};
"a0"; "b1";**\dir{-}? >* \dir{>};
"b4"; "a4";**\dir{-}? >* \dir{>};
"a3"; "b4";**\dir{-}? >* \dir{>};
"b3"; "a3";**\dir{-}? >* \dir{>};
"a2"; "b3";**\dir{-}? >* \dir{>};
"b2"; "a2";**\dir{-}? >* \dir{>};
"a1"; "b2";**\dir{-}? >* \dir{>};
"b1"; "a1";**\dir{-}? >* \dir{>};
\end{xy}
$$\\
\caption{The diagram $\mathbf{W}_m(a, b).$}
\label{wedge}
\end{figure}

 We let $\mathbf{W}_m(a, b)$ be the version of the diagram from Figure \ref{wedge} with $m$ entries to a row, $a-1$ rows above the change in direction, and $b-1$ rows after.  Lemmas \ref{comp1} and \ref{comp2} then imply that the fillings of this diagram with $K \geq \eta^k_1 - \eta^{k-1}_m, \lambda^k_1 - \lambda^{k-1}_m$ and $\eta^1_m = \lambda^1_m = 0$ are in bijection with the level $K$ piece of the graded semigroup $\bar{Q}(a, b) = Q(B, P, B)\times_{\Lambda_m}\cdots \times_{\Lambda_m} Q(B, P, B)\times_{\Lambda_m} Q(B, P^*, B) \times_{\Lambda_m}\cdots \times_{\Lambda_m} Q(B, P^*, B).$  This is the affine semigroup $Q(a, b)$ without the semigroups $Q(P, P, B)$ and $Q(B, P^*, P^*)$ at the ends of the fiber product.

Let $V_{0, 3}(r_1\mathbf{\omega}_1, r_2\mathbf{\omega}_1, \eta, K) \neq 0$, and let $\mathbf{b}$ be the associated interlacing pattern from Proposition \ref{kpieri}.  The top row of $\mathbf{b}$ is $r_1\mathbf{\omega}_{m-1} + \frac{1}{m}(r_2 + \sum \eta_i -(m-1)r_1)\mathbf{\omega}_m$, in particular the first $m-1$ entries are equal.   Proposition \ref{dualkpieri} implies an almost identical description holds for an interlacing pattern $\mathbf{b}'$ which represents a non-zero space $V_{0, 3}(\lambda, s_1\mathbf{\omega}_{m-1}, s_2\mathbf{\omega}_{m-1}, K)$, in particular the first $m-1$ entries of the bottom row of this diagram are all equal.

Now we can give a description of $Q(a, b)$ as an affine semigroup of fillings of an interlacing pattern.   An admissable weighting of $\mathbf{W}_m(a, b)$ of level $K$ is two choices of sets of interlacing weights $\lambda^1 \prec \cdots \prec \lambda^a$, $\eta^1 \prec \cdots \prec \eta^b$ with $\lambda^1 = \eta^1$, $\lambda^1_m = \eta^1_m = 0$, $\lambda^a_1 = \cdots = \lambda^a_{m-1}$, $\eta^b_1 = \cdots = \eta^b_{m-1},$ and $\lambda^i_1 - \lambda^{i-1}_m$, $\eta^j_1 - \eta^{j-1}_m \leq K.$

\begin{proposition}\label{interlace1}
The elements of $Q(a, b)$ of level $K$ are in bijection with the admissable weightings of $\mathbf{W}_m(a, b)$ of level $K$, and the product operation on $Q(a, b)$ is entry-wise addition of admissable weightings.  
\end{proposition}

\begin{proof}
By the remarks above, Lemmas \ref{comp1} and \ref{comp2} imply that the graded affine semigroup of admissable weightings is a subsemigroup of $\bar{Q}(a, b)$.  The boundary conditions  $\lambda^a_1 = \cdots = \lambda^a_{m-1}$, $\eta^b_1 = \cdots = \eta^b_{m-1},$ then imply that this is precisely the subsemigroup of those weightings with top boundary weight a multiple of $\mathbf{\omega}_1$ and bottom boundary weight a multiple of $\mathbf{\omega}_{m-1}$, this is $Q(a, b).$ 
\end{proof}

Proposition \ref{interlace1} shows that the degree $K = 1$ elements in $Q(a, b)$ are the lattice points in a polytope $\mathcal{W}_b(a, b)$ defned by the interlacing and level inequalities.  This polytope is normal by Proposition \ref{kos}, and by \cite[Corollaries 8.4 and 8.9]{St} it has a unimodular triangulation. In particular, Proposition \ref{interlace1} implies that the semigroup $Q(a, b)$ is isomorphic to the affine semigroup of of integer points in the cone over $\mathcal{W}_m(a, b) \times \{1\} \subset \R_{\geq 0}^{m(a + b -1)}\times \R_{\geq 0}.$  The same arguments in this section, omitting any mention of the level, also prove the following proposition, we leave the details to the reader. 

\begin{proposition}
The affine semigroup $L(a, b)$ is isomorphic to the affine semigroup of admissable weightings of $\mathbf{W}_m(a, b)$. 
\end{proposition}

Now we use Proposition \ref{interlace1} to show that $V_{\tree_0}(a, b)$ is a Gorenstein algebra.  Both $V_{\tree_0}(a, b)$ and $V_{\mathbb{P}^1, \vec{p}}(a, b)$ are graded domains, and these algebras share the same Hilbert function, so by \cite[Theorem 6.1]{stanley78} this shows that $V_{\mathbb{P}^1, \vec{p}}(a, b)$ is Gorenstein as well.

\begin{proposition}\label{interlacegor}
The affine semigroup algebra $\C[Q(a, b)] \cong V_{\tree_0}(a, b)$ is Gorenstein.  If $a, b \geq m$, the $A$-invariant is equal to $-2m$.
\end{proposition}

\begin{proof}
Following Proposition \ref{gor}, we must show that there is an interlacing pattern $\mathbf{w}$ such that for any interior interlacing pattern $\mathbf{b} \in int(Q(a, b))$ we can write $\mathbf{b} = \mathbf{w} + \mathbf{b}'$ for some $\mathbf{b}' \in Q(a, b).$ We must first describe the set $int(Q(a, b))$.
The inequalities which define $Q(a, b)$ come in three families: all entries except $\lambda^1_m = 0$ are non-negative, successive rows interlace, and the level inequalities $\lambda^i_1 - \lambda^{i-1}_m$, $\eta^i_1 - \eta^{i-1}_m \leq K$ must hold.  Those inequalities describe $Q'(a, b)$, the subsemigroup $Q(a, b)$ is then cut out by the equalities $\lambda^a_1 = \cdots = \lambda^a_{m -1}$ and $\eta^b_1 = \cdots = \eta^b_{m-1}$.  Interlacing then imposes further equalities:  $\lambda^a_1 = \lambda^{a-j}_1 = \cdots = \lambda^{a-j}_{m - j -1}$ and $\eta^b_1 = \eta^{b -j}_1 = \cdots = \eta^{b-j}_{m- j}$ for $1 \leq j \leq m-1.$  An example of the pattern induced by these inequalities is shown in Figure \ref{tripattern}.   The entries subject to these equalities always form two (possibly overlapping) height $m-1$ triangles in the interlacing pattern, we refer to them as the entries of the upper, respectively lower equality triangles from now on.

\begin{figure}[htbp]
\begin{tikzpicture}
%The composite interlacing pattern,first line of arrows.
\draw [->] (0,0) -- (-.8, 1.2);
\draw [->] (-1,1.2) -- (-1.8, 0);
\draw [->] (-2, 0) -- (-2.8, 1.2);
\draw [->] (-3,1.2) -- (-3.8, 0);
\draw [->] (-4, 0) -- (-4.8, 1.2);
\draw [->] (-5,1.2) -- (-5.8, 0);
\draw [->] (-6, 0) -- (-6.8, 1.2);
\draw [->] (-7,1.2) -- (-7.8, 0);
\draw [->] (-8, 0) -- (-8.8, 1.2);
%The composite interlacing pattern, second line of arrows.
\draw [->] (-1,1.6) -- (-1.8, 2.8);
\draw [->] (-2,2.8) -- (-2.8, 1.6);
\draw [->] (-3,1.6) -- (-3.8, 2.8);
\draw [->] (-4,2.8) -- (-4.8, 1.6);
\draw [->] (-5,1.6) -- (-5.8, 2.8);
\draw [->] (-6,2.8) -- (-6.8, 1.6);
\draw [->] (-7,1.6) -- (-7.8, 2.8);
\draw [->] (-8,2.8) -- (-8.8, 1.6);
\draw [->] (-9,1.6) -- (-9.8, 2.8);
%The composite interlacing pattern, third line of arrows.
\draw [->] (-2,3.2) -- (-2.8, 4.4);
\draw [->] (-3,4.4) -- (-3.8, 3.2);
\draw [->] (-4,3.2) -- (-4.8, 4.4);
\draw [->] (-5,4.4) -- (-5.8, 3.2);
\draw [->] (-6,3.2) -- (-6.8, 4.4);
\draw [->] (-7,4.4) -- (-7.8, 3.2);
\draw [->] (-8,3.2) -- (-8.8, 4.4);
\draw [->] (-9,4.4) -- (-9.8, 3.2);
\draw [->] (-10,3.2) -- (-10.8, 4.4);
%first row, right to left
\node at (.1, -.2) {$0$};
\node at (-1.9, -.2) {$2$};
\node at (-3.9, -.2) {$4$};
\node at (-5.9, -.2) {$6$};
\node at (-7.9, -.2) {$7$};
%2nd row, right to left
\node at (-.9, 1.4) {$1$};
\node at (-2.9, 1.4) {$3$};
\node at (-4.9, 1.4) {$5$};
\node at (-6.9, 1.4) {$7$};
\node at (-8.9, 1.4) {$7$};
%third row, right to left
\node at (-1.9, 3) {$2$};
\node at (-3.9, 3) {$4$};
\node at (-5.9, 3) {$7$};
\node at (-7.9, 3) {$7$};
\node at (-9.9, 3) {$7$};
%fourth row, right to left
\node at (-2.9, 4.6) {$3$};
\node at (-4.9, 4.6) {$7$};
\node at (-6.9, 4.6) {$7$};
\node at (-8.9, 4.6) {$7$};
\node at (-10.9, 4.6) {$7$};
%formula in figure
%\node at (-9, 4) {$\bar{\eta}_5 = \bar{\mu}_5 + \frac{1}{5}(r_2 + \sum \mu_i - \sum \eta_j)$};
%\node at (-3, 4){$\bar{\lambda}_5 = \bar{\eta}_5 + \frac{1}{5}(r_1 + \sum \eta_i - \sum \lambda_j)$};
\end{tikzpicture}
\label{tripattern}
\caption{The entries on the left are forced to be equal by the interlacing inequalities.}
\end{figure}
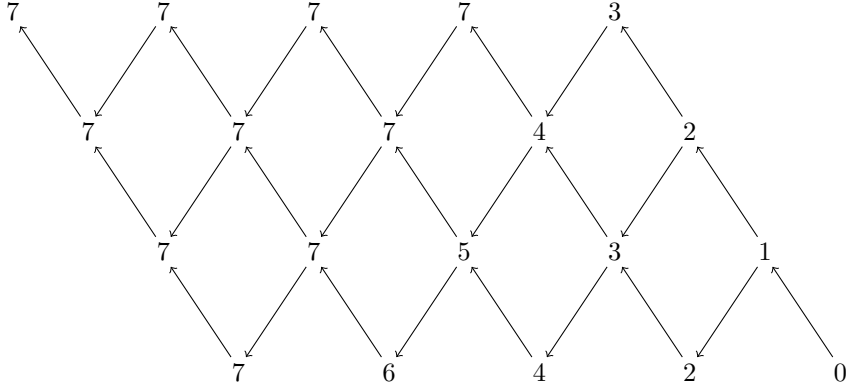

We show that the interior $int(Q(a, b))$ is given by those points which make all the above-mentioned inequalities strict, this assertion follows from finding one such point.   We define this point $\mathbf{w}$ by first considering $\mathbf{v} \in \bar{Q}(a, b)$ defined by setting $\lambda^1_m = \eta^1_m = 0$, and declaring each interlacing difference $\lambda^j_i - \lambda^{j-1}_i$, $\lambda^{j-1}_i - \lambda^{j-1}_{i+1}$, $\eta^j_i - \eta^{j-1}_i$, $\eta^{j-1}_i - \eta^{j-1}_{i+1}$ to be $1.$  The level of $\mathbf{v}$ is set equal to $2m$, as this is precisely $1$ greater than any difference $\lambda^j_1 - \lambda^{j-1}_m$. The point $\mathbf{v}$ must be in the interior of $\bar{Q}(a, b)$ as it strictly satisfies all of the defining inequalities of this affine semigroup. Furthermore, any point $\mathbf{b} \in int(\bar{Q}(a, b))$ must have all interlacing differences at least $1,$ and therefore level at least $2m$, this implies that $\mathbf{b} = \mathbf{v} + \mathbf{b}'$ for some $\mathbf{b}' \in \bar{Q}(a, b).$  With this case as a warm-up, we define $\mathbf{w}$ to have the same entries as $\mathbf{v}$, except with the equality triangle conditions imposed.  All entries in one of these triangles are set to be $1$ bigger than the largest entry with an arrow pointing ``in'' to the triangle (once again, see Figure \ref{tripattern}).  The level of $\mathbf{w}$ is chosen to be $1$ greater than the maximal difference $\lambda^j_1 - \lambda^{j-1}_m$.  

First we claim that $\mathbf{w} \in int(Q(a, b))$.  Because of our choices, checking the strictness of all inequalities reduces to only checking those interlacing inequalities involving an entry $x$ outside an equality triangle, and an entry $\tau$ inside an equality triangle. By the way we've chosen $\mathbf{w}$,  if there is an arrow from $x$ to $\tau$, $x - \tau \geq 1$.  If the arrow points from $\tau$ to $x$ (this happens if $a$ or $b$ is strictly less than $m-1$), since the triangle has base length $m-1$, and the length of each row is $m$, we can always find a directed chain of arrows from some entry with an arrow pointing into the equality triangle to $x$, so once again $x - \tau \geq 1$.  Now if $\mathbf{b} \in int(Q(a, b)),$ each difference we have mentioned above is $\geq 1$.    Since $\mathbf{w}$ is determined by forcing these differences to be $1$, the entries of $\mathbf{w}$ and the level must be less than those of $\mathbf{b}$, and the corresponding differences for the pattern $\mathbf{b} - \mathbf{w}$ must be $\geq 0$, so $\mathbf{b} - \mathbf{w} \in Q(a, b).$  

We close the proof by determining the $A$-invariant in the case $a, b \geq m$. In this case the $\lambda^1 = \eta^1 = (2(m-1), \ldots, 2, 0)$ and $\lambda^2 = (2m-1, 2m-3, \ldots, 1)$, so the level of $\mathbf{w}$ is $K = 2m.$
\end{proof}

\begin{remark}\label{2gen}
The program we have carried out for $Q(a, b)$ goes through with little modification for $\bar{Q}(a, b)$. Since the affine semigroup algebra $\C[\bar{Q}(a, b)]$ is a degeneration of $V_{\mathbb{P}^1, \vec{p}}(B, \vec{P}, \vec{P}^*, B)$, Theorem \ref{main} holds for the total coordinate ring of $\mathcal{M}_{\mathbb{P}^1, \vec{p}}(B, \vec{P}, \vec{P}^*, B)$ as well. 
\end{remark}

\subsection{Presentations of $V_{\tree_0}(a, b)$ and $R_{\tree_0}(a, b)$}

Proposition \ref{kos} implies that $Q(a, b)$ is generated by tuples $(\mathbf{w}_1, \ldots, \mathbf{w}_{a + b -1})$ of degree $1$ generators of the $K$-Pieri semigroups with matching boundary values.  The generators of $L(a, b)$ are similar, except $\mathbf{0}$ is allowed to appear.   By Propositions \ref{4algebrasclassical} and \ref{4algebrasquantum}, these generators are the interlacing patterns filled with $1$'s and $0$'s.  In $Q(a, b)$, each of the patterns has $K = 1$, and there is an additional empty pattern of level $1$, not to be confused with the identity $\mathbf{0}$. We introduce a new notation $[i, j]$ for these generators, this indicates a pattern with top (bottom) rows consisting of $i$ ($j$) $1$'s respectively, see  Figure \ref{pieripattern2}.

\begin{figure}[htbp]
\begin{tikzpicture}
%The top interlacing pattern 
%3 zigzags
\draw [->] (0,3) -- (-.8, 4.2);
\draw [->] (-1,4.2) -- (-1.8, 3);
\draw [->] (-2, 3) -- (-2.8, 4.2);
%3 zigzags
\draw [->] (-4, 3) -- (-4.8, 4.2);
\draw [->] (-5,4.2) -- (-5.8, 3);
\draw [->] (-6, 3) -- (-6.8, 4.2);
%3 zigzags
\draw [->] (-8, 3) -- (-8.8, 4.2);
\draw [->] (-9, 4.2) -- (-9.8, 3);
\draw [->] (-10, 3) -- (-10.8, 4.2);
%first zigzag bottom
\node at (.1, 3-.2) {$0$};
\node at (-1.9, 3-.2) {$0$};
%first zigzag top
\node at (-.9, 4.4) {$0$};
\node at (-2.9, 4.4) {$0$};
%second zigzag bottom
\node at (-3.9, 3-.2) {$0$};
\node at (-5.9, 3-.2) {$1$};
%second zigzag top
\node at (-4.9, 4.4) {$0$};
\node at (-6.9, 4.4) {$1$};
%third zigzag bottom
\node at (-7.9, 3-.2) {$1$};
\node at (-9.9, 3-.2) {$1$};
%third zigzag top
\node at (-8.9, 4.4) {$1$};
\node at (-10.9, 4.4) {$1$};
%ellipses
\node at (-3.4, 3.6) {$\cdots$};
\node at (-7.4, 3.6) {$\cdots$};
%The bottom interlacing pattern 
%3 zigzags
\draw [->] (0,0) -- (-.8, 1.2);
\draw [->] (-1,1.2) -- (-1.8, 0);
\draw [->] (-2, 0) -- (-2.8, 1.2);
%3 zigzags
\draw [->] (-4,0) -- (-4.8, 1.2);
\draw [->] (-5,1.2) -- (-5.8, 0);
\draw [->] (-6, 0) -- (-6.8, 1.2);
%3 zigzags
\draw [->] (-8,0) -- (-8.8, 1.2);
\draw [->] (-9,1.2) -- (-9.8, 0);
\draw [->] (-10, 0) -- (-10.8, 1.2);
%first zigzag bottom
\node at (.1, -.2) {$0$};
\node at (-1.9, -.2) {$0$};
%first zigzag top
\node at (-.9, 1.4) {$0$};
\node at (-2.9, 1.4) {$0$};
%second zigzag bottom
\node at (-3.9, -.2) {$0$};
\node at (-5.9, -.2) {$1$};
%second zigzag top
\node at (-4.9, 1.4) {$0$};
\node at (-6.9, 1.4) {$1$};
%third zigzag bottom
\node at (-7.9, -.2) {$1$};
\node at (-9.9, -.2) {$1$};
%third zigzag top
\node at (-8.9, 1.4) {$1$};
\node at (-10.9, 1.4) {$1$};
%ellipses
\node at (-3.4, .6) {$\cdots$};
\node at (-7.4, .6) {$\cdots$};
\end{tikzpicture}
\caption{The elements $[i, i]$ (top) and $[i+1, i]$ (bottom), $1 \leq i \leq m-1$}
\label{pieripattern2}
\end{figure}
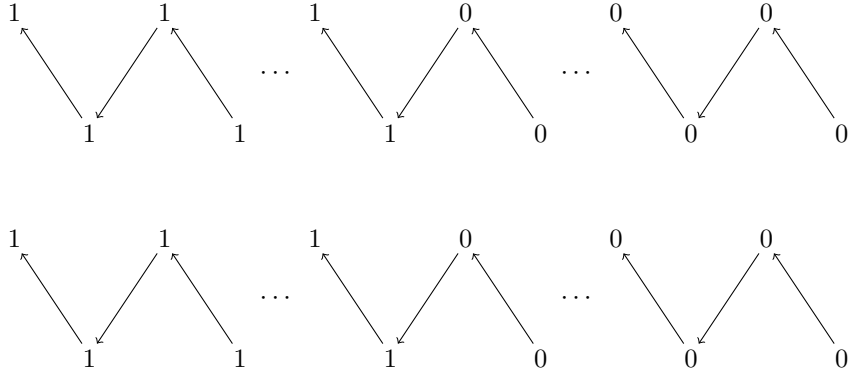

\noindent
For any generator $[i, j],$ the difference $| i - j |$ is $1$ or $0$, and a row of all $1$'s is equivalent to a row of $0$'s under $\del_1, \del_2$, so we consider $i, j$ modulo $m$.  This implies that a generating set of $Q(a, b)$ and $L(a, b)$ is given by a set of $a +b -1$ tuples of elements from $\Z/m\Z.$

\begin{definition}
Let $X_{a, b}$ be the set of $a + b -1$-tuples $[i_1, \ldots, i_{a + b -1}]$,  $i_j \in \Z/m\Z$ with the following properties:

\begin{enumerate}
\item The difference $i_{k} - i_{k+1}$ is in $\{1, 0\}$ for $1 \leq k \leq a-1$.\\
\item The difference $i_{k} - i_{k+1}$ is in $\{-1, 0\}$ for $a \leq k \leq a + b -2$.\\
\item The first and last entries are members of $\{m-1, m\}.$\\
\end{enumerate}

Let $Y_{a, b} \subset X_{a, b}$ be the subset of tuples consisting of at most one string of non-zero entries. 
\end{definition}

A swap relation on the set $X_{a, b}$ can be formed whenever a pair of elements of this set share a common entry:

\begin{equation}
[\vec{i}_1, i, \vec{i_2}][\vec{i_1}^*, i, \vec{i_2}^*] - [\vec{i}_1, i, \vec{i_2}^*][\vec{i_1}^*, i, \vec{i_2}]\\
\end{equation}

\noindent
The next theorem is a direct consequence of Proposition \ref{kos}. 

\begin{theorem}\label{pres}
The semigroup algebra $V_{\tree_0}(a, b) \cong \C[Q(a, b)]$ can be presented by the generators $X_{a, b}$, subject to all possible swap relations. The semigroup algebra $R_{\tree_0}(a, b) \cong \C[L(a, b)]$ can be presented by the generators $Y_{a, b}$, subject to all possible swap relations. 
\end{theorem}

\begin{proof}
A-priori the $X_{a, b}$ generates $L(a, b)$ as well, so it only remains to show that the set $Y_{a, b}$ suffices.  This is the case because we may use the swap relations which hold when two generators share a $0$ entry to reduce an element of $X_{a, b}$ to a product of elements of $Y_{a, b}.$ 

\begin{equation}
[\vec{i}_1, 0, \vec{i}_2][0, \ldots, 0, \ldots, 0] = [\vec{i}_1, 0, \ldots, 0][0, \ldots, 0, \vec{i}_2]\\
\end{equation}

\end{proof}

\subsection{Comparison of presentations}

The algebra $R_{\tree_0}(a, b)$ is an associated graded algebra of $R(a, b)$, this implies that our presentation of $R_{\tree_0}(a, b)$ can be lifted to $R(a, b).$   We close by comparing the generating set $Y_{a, b}$ to the generators given by the first fundamental theorem of invariant theory.  We view the algebra $\C[[\bigwedge^{m-1}(\C^m)]^a \times [\C^m]^b]$ as a polynomial ring on $m \times (a +  b)$ variables arranged in a matrix, with the columns labelled by elements of $[a + b] = [a] \coprod [b]$. 

\[ \begin{array}{cccccc}
x_{11} & \ldots & x_{a1} & y_{11} & \ldots & y_{b1}\\
x_{12} & \ldots & x_{a2} & y_{12} & \ldots & y_{b2}\\
\ldots & \ldots &  \ldots & \ldots & \ldots & \ldots \\
x_{1m} & \ldots & x_{am} & y_{1m} & \ldots & y_{bm}\end{array}\] 

 We let $\Delta_I$ be the determinant form on the variables determined by a subset $I \subset [a]$ in $[\bigwedge^{m-1}(\C^m)]^a$ of size $m$, and $\Delta_J$ be the dual determinant in $[\C^m]^b,$ for $J \subset [b].$  We let $P_{ij}$ be the column-wise inner product of the variables on the indices $i \in [a]$ with those in $j \in [b].$ Each of the elements $\Delta_I, \Delta_J, P_{ij}$ is an $\mathrm{SL}_m(\C)$ invariant, and therefore a member of $R(a, b)$.  Weyl described a collection of quadratic relations on these elements, known as the  Pl\"ucker relations. The following is the first fundamental theorem:

\begin{theorem}\label{ffti}
The algebra $R(a, b)$ is generated by the $\Delta_I, \Delta_J, P_{ij}$, and all relations
among these generators are generated by the Pl\"ucker relations. 
\end{theorem}

By reading off the fundamental weights of each of the generators $\Delta_I$, $\Delta_J$, and $P_{ij}$ from Weyl's presentation, we can associate minimal generators in the degeneration $R_{\tree_0}(a, b),$ as follows.   For the set $I \subset [a],$ we let $t_I$ be the tuple with $t_{i+1} - t_{i} = 1$ exactly for $i \in I,$ and we define $t_J$ for $J \subset [b]$ similarly.  We let $t_{ij}$ be the tuple with $t_{i+1} - t_{i} = 1$, $t_{j+1} - t_j = -1$,and all other differences equal to $0$. These elements constitute a proper subset of $Y_{a, b}$.

Theorem \ref{pres} implies that the induced presentation of $R(a, b)$ by the set $Y_{a, b}$ has quadratic relations, similar to the presentation given by the first fundamental theorem.  The latter requires fewer generators, however their associated elements fail to generate $R_{\tree_0}(a, b)$.  Said more concisely, the set $Y_{a, b}$ constitutes a SAGBI basis of the algebra $R(a, b)$ with respect to the degeneration to the semigroup algebra $R_{\tree_0}(a, b)$.

\bibliographystyle{alpha}
\bibliography{Biblio}

\bigskip
\noindent
Christopher Manon:\\
Department of Mathematics,\\ 
George Mason University,\\ 
Fairfax, VA 22030 USA

\end{document}